\documentclass[a4paper,reqno,11pt]{amsart}
\usepackage{amsfonts,amsmath,amsthm,amssymb,stmaryrd}
\newtheorem{theorem}{Theorem}[section]
\newtheorem{lemma}[theorem]{Lemma}

\newtheorem{Remark}[theorem]{Remark}
\numberwithin{equation}{section}

\usepackage[left=1 in, right=1 in,top=1 in, bottom=1 in]{geometry}

 \providecommand{\norm}[1]{\left\Vert#1\right\Vert}

\def\r3{\mathbb{R}^3}

\begin{document}
\title[Decay of VPB]{Decay of the Vlasov-Poisson-Boltzmann system}

\author{Yanjin Wang}
\address{
School of Mathematical Sciences\\
Xiamen University\\
Xiamen, Fujian 361005, China} \email{yanjin$\_$wang@xmu.edu.cn}
\thanks{This work is partially supported by National Natural Science Foundation of China-NSAF (No. 10976026)}

\begin{abstract}
We establish the time decay rates of the solutions to the Cauchy problem for the two-species Vlasov-Poisson-Boltzmann system near Maxwellians via a
refined pure energy method. The negative Sobolev norms are shown to be preserved along time evolution and enhance the decay rates. The total density of
two species of particles decays at the optimal algebraic rate as the Boltzmann equation, but the disparity between two species and the electric field decay at an
exponential rate. This phenomenon reveals the essential difference when compared to the one-species Vlasov-Poisson-Boltzmann system in which the
electric field decays at the optimal algebraic rate or the Vlasov-Boltzmann system in which the disparity between two species decays at the optimal
algebraic rate. Our proof is based on a family of scaled energy
estimates with minimum derivative counts and interpolations among them
without linear decay analysis, and a reformulation of the problem which well displays the cancelation property of the two-species system.
\end{abstract}
\maketitle


\section{Introduction}
The dynamics of charged dilute particles (e.g., electrons and ions) in the absence of magnetic effects can be described by the
Vlasov-Poisson-Boltzmann system:
\begin{equation}\label{VPB0}
\begin{split}
&\partial_tF_++v\cdot\nabla_xF_++\nabla_x\Phi\cdot\nabla_vF_+=Q(F_+,F_+)+Q(F_+,F_-),
\\&\partial_tF_-+v\cdot\nabla_xF_--\nabla_x\Phi\cdot\nabla_vF_-=Q(F_-,F_+)+Q(F_-,F_-),
\\&\Delta_x\Phi=\int_{\r3}F_+-F_-\,dv,
\end{split}
\end{equation}
with initial data $F_\pm(0,x,v)=F_{0,\pm}(x,v)$. Here $F_\pm(t,x,v)\ge 0$ are the number density functions for the ions $(+)$ and electrons $(-)$
respectively, at time $t\ge 0$, position $x=(x_1,x_2,x_3)\in \mathbb{R}^3$ and velocity $v=(v_1,v_2,v_3)\in \mathbb{R}^3$. The self-consistent
electric potential $\Phi=\Phi(t,x)$ is coupled with $F_\pm(t,x,v)$ through the Poisson equation. The collision between particles is given by the
standard Boltzmann collision operator $Q(h_1,h_2)$ with hard-sphere interaction:
\begin{equation}\label{Boltzmann operator}
Q(h_1,h_2)(v)=\int_{\mathbb{R}^3}\int_{\mathbb{S}^2}|(u-v)\cdot\omega|\{h_1(v')h_2(u')-h_1(v)h_2(u)\}\, d\omega\, du.
\end{equation}
Here $\omega\in \mathbb{S}^2$, and
\begin{equation}
v'=v-[(v-u)\cdot\omega]\omega,\quad u'=u+[(v-u)\cdot\omega]\omega,
\end{equation}
which denote velocities after a collision of particles having velocities $v$ and $u$ before the collision, and vice versa. Since the presence of all the physical
constants does not create essential mathematical difficulties, for notational simplicity, we have normalized all constants in the
Vlasov-Poisson-Boltzmann system to be one. Accordingly, we normalize the global Maxwellian as (with $\nabla_x\Phi\equiv 0$)
\begin{equation}
\mu(v)\equiv\mu_+(v)=\mu_-(v)= \frac{1}{(2\pi)^{3/2}}{\rm e}^{-|v|^2/2}.
\end{equation}

In this paper, it is more convenient to consider the sum and difference of $F_+$ and $F_-$. This is motivated by some previous works
\cite{BELM,J2009,W} on the other systems of binary fluids. Defining
\begin{equation}
F\equiv F_++F_-\ \hbox{ and  }\ G\equiv F_+-F_-,
\end{equation}
that is, $F$ is the total density for the two species of particles and $G$ represents the disparity between two species. Then the system \eqref{VPB0}
can be written as the equivalent form:
\begin{equation}\label{VPB1}
\begin{split}
&\partial_tF+v\cdot\nabla_xF+\nabla_x\Phi\cdot\nabla_v G=Q(F,F),
\\&\partial_tG+v\cdot\nabla_xG+\nabla_x\Phi\cdot\nabla_vF=Q(G,F),
\\&\Delta_x\Phi=\int_{\r3}G\,dv,
\end{split}
\end{equation}
with initial data $F(0,x,v)=F_{0}(x,v)$ and $G(0,x,v)=G_{0}(x,v)$. We define the standard perturbation $[f(t,x,v),g(t,x,v)]$ around the corresponding
equilibrium state $[\mu,0]$ as
\begin{equation}
F=\mu+\sqrt{\mu}f\ \hbox{ and }\ G=\sqrt{\mu}g,
\end{equation}
then the Vlasov-Poisson-Boltzmann system for the perturbation $[f,g]$ takes the form
\begin{equation}\label{VPB_per}
\begin{split}
&\partial_tf + v\cdot\nabla_xf+ \mathcal{L}_1 f=\mathfrak{N}_1:=\Gamma(f,f)+\frac{1}{2}\nabla_x\Phi\cdot vg-\nabla_x\Phi\cdot\nabla_vg ,
\\ &\partial_tg + v\cdot\nabla_xg-\nabla_x\Phi\cdot v\sqrt{\mu} + \mathcal{L}_2 g=\mathfrak{N}_2:=\Gamma(g,f)+\frac{1}{2}\nabla_x\Phi\cdot vf-\nabla_x\Phi\cdot\nabla_vf ,
\\ &\Delta_x\Phi=\int_{\r3}g\sqrt{\mu}\,dv,
\end{split}
\end{equation}
with initial data $f(0,x,v)=f_{0}(x,v)$ and $g(0,x,v)=g_{0}(x,v)$. Here $\mathfrak{N}_1$ and $\mathfrak{N}_2$ represent the nonlinear terms. The
well-known linearized collision operator $\mathcal{L}_1$ and another linearized operator $\mathcal{L}_2$ are defined as
\begin{equation}
\mathcal{L}_1h=-\frac{1}{\sqrt{\mu}}\{Q(\mu,\sqrt{\mu}h)+Q(\sqrt{\mu}h,\mu)\},\quad \mathcal{L}_2h=-\frac{1}{\sqrt{\mu}}Q(\sqrt{\mu}h,\mu),
\end{equation}
and the nonlinear collision operator (non-symmetric) is given by
\begin{equation}
\Gamma(h_1,h_2)=\frac{1}{\sqrt{\mu}}Q(\sqrt{\mu}h_1,\sqrt{\mu}h_2).
\end{equation}
Note that the linear  homogeneous system of \eqref{VPB_per} is decoupled into two independent subsystems. One is the Boltzmann equation for $f$,
and the other one is a system almost like the Vlasov-Poisson-Boltzmann for $g$ and $\nabla_x\Phi$ but with a different linearized collision operator
$\mathcal{L}_2$.

Notice that $[\mathcal{L}_1,\mathcal{L}_2]$ is equivalent to the linearized collision operator $L$ defined in \cite{G2003}. It is well-known that the
operators $\mathcal{L}_1$ and $\mathcal{L}_2$ are non-negative. For any fixed $(t,x)$, the null spaces of $\mathcal{L}_1$ and $\mathcal{L}_2$ are
given by, respectively,
\begin{equation}
\mathcal{N}(\mathcal{L}_1)={\rm span}\{\sqrt{\mu},v\sqrt{\mu},|v|^2\sqrt{\mu}\}\ \hbox{ and }\ \mathcal{N}(\mathcal{L}_2)={\rm span}\{\sqrt{\mu}\}.
\end{equation}
For any fixed $(t,x)$, we define ${\bf P_1}$ as the $L^2_v$ orthogonal projection on the null space $\mathcal{N}(\mathcal{L}_1)$. Thus for any
function $f(t,x,v)$  we can decompose
\begin{equation}
f={\bf P_1}f+\{{\bf  I-P_1}\}f,
\end{equation}
where ${\bf P_1}f$ is called the hydrodynamic part of $f$ and $\{{\bf I- P_1}\}f$ is the microscopic part. We can
further denote
\begin{equation}\label{hydrodynamic field 1}
{\bf P_1}f=\left\{a(t,x)+b(t,x)\cdot v+c(t,x)\left(\frac{|v|^2}{2}-\frac{3}{2}\right)\right\}\sqrt{\mu}.
\end{equation}
Here the hydrodynamic field of $f$, $[a(t,x),b(t,x),c(t,x)]$, represents the density, velocity and temperature fluctuations physically.
Similarly, for any fixed $(t,x)$, we define ${\bf P_2}$ as the $L^2_v$ orthogonal projection on the null space $\mathcal{N}(\mathcal{L}_2)$. Thus for
any function $g(t,x,v)$ we can decompose
\begin{equation}
g={\bf P_2}g+\{{\bf I-P_2}\}g,
\end{equation}
and we can further denote
\begin{equation}\label{hydrodynamic field 2}
{\bf P_2}g=d(t,x)\sqrt{\mu}.
\end{equation}
Here $d(t,x)$ represents the concentration difference fluctuation.

\smallskip

\noindent{\bf Notation.} In this paper, $\nabla^\ell$ with an integer $\ell\ge0$ stands for the usual any spatial derivatives of order $\ell$. When
$\ell<0$ or $\ell$ is not a positive integer, $\nabla^\ell$ stands for $\Lambda^\ell$ defined by \eqref{1Lambdas}.  We use $\dot{H}^s(\mathbb{R}^3),
s\in \mathbb{R}$ to denote the homogeneous Sobolev spaces on $\mathbb{R}^3$ with norm $\norm{\cdot}_{\dot{H}^s}$ defined by \eqref{1snorm}, and we
use $H^s(\mathbb{R}^3) $ to denote the usual Sobolev spaces with norm $\norm{\cdot}_{H^s}$ and $L^p(\mathbb{R}^3), 1\le p\le \infty$ to denote the usual
$L^p$ spaces with norm $\norm{\cdot}_{L^p}$.

We shall use $\langle\cdot,\cdot\rangle$ to denote the $L^2$ inner product in $\mathbb{R}^3_v$ with corresponding $L^2$ norm $|\cdot|_2$, while we
use $(\cdot,\cdot)$ to denote the $L^2$ inner product either in $\mathbb{R}^3_x\times\mathbb{R}^3_v$ or in $\mathbb{R}^3_x$ with $L^2$ norm
$\|\cdot\|$ without any ambiguity. We shall simply use $L^2_x$, $L^2_v$ to denote $L^2(\mathbb{R}^3_x)$ and $L^2(\mathbb{R}^3_v)$ respectively, etc.
We will use the notation $L^2_v H^s_x$ to denote the space $L^2(\mathbb{R}^3_v; H^s_x)$ with norm
\begin{equation}
\norm{h}_{L^2_v H^s_x}=\left(\int_{\mathbb{R}^3_v}\norm{f}_{H^s_x}^2\,dv\right)^{1/2},
\end{equation}
and similarly we use the notations of $L^2_v \dot{H}^s_x$, $L^2_v L^p_x$ and $L^p_xL^2_v$, etc. We define the space-velocity mixed derivatives by
$$\partial^\gamma_\beta=\partial_{x_1}^{\gamma_1}\partial_{x_2}^{\gamma_2}\partial_{x_3}^{\gamma_3}
\partial_{v_1}^{\beta_1}\partial_{v_2}^{\beta_2}\partial_{v_3}^{\beta_3}$$
where $\gamma=[\gamma_1,\gamma_2,\gamma_3]$ is related to the space derivatives, while $\beta=[\beta_1,\beta_2,\beta_3]$ is related to the velocity
derivatives.

For the Boltzmann operator \eqref{Boltzmann operator}, we define the collision frequency as
\begin{equation}
\nu(v)=\int_{\mathbb{R}^3}|v-u|\mu(u)du,
\end{equation}
which behaves like $1+|v|$. We define the weighted $L^2$ norms
\begin{equation}
|g|_\nu^2=|\nu^{1/2}g|_2^2,\quad \|g\|_\nu^2=\|\nu^{1/2}g\|^2.
\end{equation}
We denote $L^2_\nu$ by the weighted space with norm $\|\cdot\|_\nu$.

Throughout this paper, $C>0$ will denote a generic constant that can depend on the parameters coming from the problem, and the indexes $N$ and $s$
coming from the regularity on the data, but does not depend on the size of the data, etc. We refer to such constants as ``universal.'' Such constants
are allowed to change from line to line. We will employ the notation $a \lesssim b$ to mean that $a \le C b$ for a universal constant $C>0$. We also
use $C_0$ for a positive constant depending additionally on the initial data. To indicate some constants in some places so that they can be referred
to later, we will denote them in particular by $C_1,C_2$, etc.

\smallskip

As it can be obtained by formally setting the magnetic field to be zero in the two-species Vlasov-Maxwell-Boltzmann, the global existence of
classical solutions near Maxwellians to the two-species Vlasov-Poisson-Boltzmann system \eqref{VPB0} can be found in \cite{G2003,J2009} for the
spatially periodic domain and in \cite{S2006} for the whole space. For the one-species Vlasov-Poisson-Boltzmann system, the first global unique
solution near Maxwellians was constructed in \cite{G2002} for the periodic domain and in \cite{YYZ} for the whole space under the restrictions that
either the mean free path is sufficiently small or the constant background charge density is sufficiently large. Those restrictions in \cite{YYZ}
were removed in \cite{YZ}, and \cite{DY} proved the global existence of solutions for more general nonconstant background density. These proofs are
based on the nonlinear energy method developed in \cite{G2002,G2003,G2004,G2006} and in \cite{LYY2004,LY2004}. It seems that those results on the existence
of solutions can be directly generalized to the two-species Vlasov-Poisson-Boltzmann system without any additional essential difficulties.

The time decay rate of the solutions has been an important problem. It is well known in \cite{G2003} that for the periodic domain the solutions to
the Vlasov-Poisson-Boltzmann system decay at an exponential rate similarly as the Boltzmann equation \cite{U1974,S1983}. For the whole space, \cite{YZ} also obtained a time decay rate of $O(t^{-1/2})$ in
the $L^\infty_xL^2_v$-norm of the solution for the one-species Vlasov-Poisson-Boltzmann system based on the pure energy method and a time
differential inequality, and this result was generalized to the two-species system in \cite{Z}. It is noticed that the decay rates in \cite{YZ} and
\cite{Z} are not optimal. Concerning the optimal decay rates, under the additional assumption that $L^2_vL^1_x$-norm of the initial perturbation
$f_0$ is sufficiently small, by combining the linear decay results from the Fourier analysis and the nonlinear energy estimates, \cite{DS1} proved
that the $L^2$-norm of solutions to the one-species Vlasov-Poisson-Boltzmann system decay at the optimal rate of $O(t^{-1/4})$ which is slower than
the $O(t^{-3/4})$ of the Boltzmann equation. While under the additional assumption that $L^2_vL^1_x$-norm of the initial perturbation
$f_0=[f_{0,+},f_{0,-}]$ and the $L^1_x$-norm of $\nabla_x\Phi_0$ are sufficiently small, combining the construction of the compensating functions
for two-species system by using Kawashima's method \cite{K1990} and the nonlinear energy estimates, \cite{YY} proved that the $L^2$-norm of solutions to the two-species
Vlasov-Poisson-Boltzmann system decay at the optimal rate of $O(t^{-3/4})$ which is same as that of Boltzmann equation \cite{NI1976,UY2006}. One may also find from
\cite{DS2} that the $L^2_vL^r_x$-norm with $2\le r\le \infty$ of solutions to the two-species Vlasov-Poisson-Boltzmann system decay at the optimal
rate of $O(t^{-3/2(1-1/r)})$. We remark that if given the same condition on the initial data (the difference between \cite{DS1} and \cite{YY,DS2} is
whether one imposes the condition that $\norm{\nabla_x\Phi}_{L^1_x}$ is small, which is equivalent to that $\norm{\nabla_x\Delta^{-1}\langle
f_0,\sqrt{\mu}\rangle}_{L^1_x}$ is small), \cite{DS1} and \cite{YY,DS2} should get the same optimal decay rate of the solutions. It is observed that
the electric field is coupled with $f$ through the Poisson equation, so one may obtain the algebraic decay rate of $\nabla_x\Phi$ from that of $f$ in
\cite{DS1,DS2,YY}. One may argue by contradiction that from the structure of the problem, the decay rate of electric field of the one-species
system is optimal. However, we will show in this paper that for the two-species system the electric field decays at an exponential rate. This is
totally due to the special cancelation property of the two-species system.

It is difficult to show that the $L^{p}$ norm of the solution can be
preserved along time evolution in the $L^{p}$--$L^{2}$ approach \cite{DS1,DS2,YY,NI1976,UY2006}. On the other
hand, the existing pure energy method \cite{YZ,Z} of proving the decay
rate does not lead to the optimal decay rate of the solution. Motivated by \cite{GT}, using a negative Sobolev space $\dot{H}^{-s}$ ($s\ge 0$) to replace $L^{p}$
norm, in \cite{GW} we developed a general energy method of using a family of scaled energy
estimates with minimum derivative counts and interpolations among them to prove the optimal decay rate of the dissipative equations in the whole space. The method was applied to classical
examples such as the heat equation, the compressible Navier-Stokes equations and the Boltzmann
equation.  The main purpose of this paper is to applying this energy method to prove the $L^2$ optimal decay rate of the solution to the two-species
Vlasov-Poisson-Boltzmann system \eqref{VPB0}, equivalently, the system \eqref{VPB_per} for $[f,g]$.

Now we define the equivalent instant energy functional $\mathcal{E}_N$ by
\begin{equation}
\mathcal{E}_N\sim\sum_{|\gamma|+|\beta|\le N}\norm{\left[\partial^\gamma_\beta f,\partial^\gamma_\beta g\right]}^2+\sum_{|\gamma|\le
N+1}\norm{\partial^\gamma\nabla_x\Phi}^2
\end{equation}
and the corresponding dissipation rate by
 \begin{equation}
\begin{split}
\mathcal{D}_N&=\sum_{|\gamma|+|\beta|\le N}\norm{[\partial^\gamma_\beta\{{\bf I-P_1 }\}f,\partial^\gamma_\beta\{{\bf I-P_2 }\}g]}_\nu^2
 \\&\quad+\sum_{1\le |\gamma|\le N}
\|\partial^\gamma {\bf P_1 }f \|^2+\sum_{ |\gamma|\le N} \norm{\partial^\gamma {\bf P_2 }g}^2 +\sum_{|\gamma|\le N+1} \|\partial^\gamma \nabla_x\Phi
\|^2.
\end{split}
\end{equation}
We remark that the definitions of $\mathcal{E}_N$ and $\mathcal{D}_N$ are carefully designed to capture the structure of the system \eqref{VPB_per}.
Notice that only $\norm{{\bf P_1 }f}^2$ in $\mathcal{E}_N$ is not bounded by the dissipation $\mathcal{D}_N$. We assume that $N$ is a sufficiently
large integer. Our first main result can be stated as follows.

\begin{theorem}\label{theorem1}
There exists an instant energy functional $\mathcal{E}_N(t)$ such that if $\mathcal{E}_N(0)$ is sufficiently small, then the Vlasov-Poisson-Boltzmann
system \eqref{VPB_per} admits a unique global solution $[f(t,x,v),g(t,x,v)]$ satisfying that for all $t\ge0$,
\begin{equation}\label{energy es}
\mathcal{E}_N(t)+\int_0^t\mathcal{D}_N(\tau)\,d\tau\le \mathcal{E}_N(0).
\end{equation}
If further, $f_0\in L^2_v\dot{H}^{-s}_x$ for some $s\in [0,3/2)$, then for all $t\ge 0$,
\begin{equation}\label{H-sbound}
\norm{\Lambda^{-s}f(t)}\le C_0,
\end{equation}
and the following decay results hold:
\begin{equation}\label{decay00}
\mathcal{E}_N(t)\le C_0(1+t)^{-s}
\end{equation}
and
\begin{equation}\label{decay0}
\sum_{1\le k \le N}\norm{\left[\nabla^k f(t),\nabla^k g(t)\right]}^2+\sum_{1\le k \le N+1}\norm{\nabla^k \nabla_x\Phi(t)}^2\le C_0(1+t)^{-(1+s)}.
\end{equation}
\end{theorem}

\begin{Remark}
We remark that by the Poisson equation $\eqref{VPB_per}_3$, we have
\begin{equation}
\mathcal{E}_N\sim\sum_{|\gamma|+|\beta|\le N}\norm{\left[\partial^\gamma_\beta f,\partial^\gamma_\beta g\right]}^2+ \norm{ \nabla_x\Delta^{-1}{\bf P_2}g}^2.
\end{equation}
\end{Remark}
\begin{Remark}
Notice that we do not need to assume that the $L^2_v\dot{H}^{-s}_x$ norm of $f_0$ is small. This norm is preserved along time evolution and enhances
the decay rate of the solution with the factor $s/2$. The constraint $s<3/2$ comes from applying Lemma \ref{1Riesz} to estimate the nonlinear terms
when doing the negative Sobolev estimates via $\Lambda^{-s}$. For $s\ge 3/2$, the nonlinear estimates would not work. We remark that we do not
require the $L^2_v\dot{H}^{-s}_x$ norm of $g_0$.
\end{Remark}

As we metioned before that the linear homogeneous system of \eqref{VPB_per} is decoupled into the Boltzmann equation for $f$ and another system for
$g$ and $\nabla_x\Phi$ looks like the one-species Vlasov-Poisson-Boltzmann but with the linearized collision operator $\mathcal{L}_2$. However,
indifferent from the Boltzmann equation, the Vlasov-Boltzmann equation and the one-species Vlasov-Poisson-Boltzmann system, this special coupling
effect between the Poisson equation and the linear operator $\mathcal{L}_2$ gives the dissipation estimate of not only the microscopic part $\{{\bf
I-P_2}\}g$ but also the hydrodynamic part ${\bf P_2}g$ (it only has one hydrodynamic field, namely, $d(t,x)$). This is also what we accorded to when we
define the dissipation $\mathcal{D}_N$. Then it would suggest that $f$ will decay at the optimal algebraic rate as the Boltzmann equation but $g$
(and $\nabla_x\Phi$) may decay at the exponential rate. These facts can be verified easily for the linear homogeneous system of \eqref{VPB_per}
along our proof of Theorem \ref{theorem1} or by the Fourier analysis as in \cite{DS1,DS2}. However, for the nonlinear problem \eqref{VPB_per}, to
control the nonlinear terms we need to impose a bit stronger assumption for the initial data. This is our second main result which can be stated as
follows.

\begin{theorem}\label{theorem2}
Under the assumptions of Theorem \ref{theorem1}, if additionally $\mathcal{E}_N(0)+\norm{[f_0,g_0]}_\nu^2$ is sufficiently small, then for all $t\ge
0$,
\begin{equation}\label{decay21}
\sum_{2\le k \le N}\norm{\left[\nabla^k f(t),\nabla^k g(t)\right]}^2+\sum_{2\le k \le N+1}\norm{\nabla^k \nabla_x\Phi(t)}^2\le C_0(1+t)^{-(2+s)};
\end{equation}
if additionally $\mathcal{E}_N(0)+\sum_{0\le k\le N-1}\norm{[\nabla^kf_0,\nabla^kg_0]}_\nu^2$ is sufficiently small, then for all $t\ge 0$,
\begin{eqnarray}\label{decay22}
&&\sum_{0\le k \le N-1}\norm{\nabla^k g(t)}^2+\sum_{0\le k \le N}\norm{\nabla^k \nabla_x\Phi(t)}^2\le C_0e^{-\lambda t}\ \text{ for some }\lambda>0,
\\&&\label{decay231} \norm{\nabla^N
g(t)}^2+\norm{\nabla^{N+1} \nabla_x\Phi(t)}^2\le C_0(1+t)^{-(N-1+s)},
\end{eqnarray}
and
\begin{eqnarray}\label{decay23}
&&\sum_{\ell\le k \le N}\norm{\nabla^k f(t)}^2\le C_0(1+t)^{-(\ell+s)}\ \text{ for }\ \ell=3,\dots,N-1,
\\&&\label{decay24}
\norm{\nabla^k \{{\bf I-P_1 }\} f(t)}^2\le C_0(1+t)^{-(k+1+s)}\, \hbox{ for }k=0,\dots,N-2.
\end{eqnarray}
\end{theorem}

We will prove Theorem \ref{theorem1} and Theorem \ref{theorem2} by the energy method that we recently developed in \cite{GW}. As there, we may use the linear heat equation to illustrate the main idea of this approach in advance. Let $u(t)$ be the solution to the heat equation
\begin{equation}\label{heat equation}
\left\{\begin{array}{lll}
\partial_tu -\Delta u=0\ \text{ in }\mathbb{R}^3
\\u|_{t=0}=u_0,
\end{array}\right.
\end{equation}
Let $-s\le \ell\le N$. The standard energy identity of \eqref{heat equation} is
\begin{equation}\label{heat 1}
\frac{1}{2}\frac{d}{dt}\norm{\nabla^\ell u}_{L^2}^2+\norm{\nabla^{\ell+1} u}_{L^2}^2=0.
\end{equation}
Integrating the above in time, we obtain
\begin{equation}\label{heat 2}
\norm{\nabla^\ell u(t)}_{L^2}^2\le \norm{\nabla^\ell u_0}_{L^2}^2.
\end{equation}
Note that the energy in \eqref{heat 1} is not bounded by the corresponding dissipation. But the crucial observation is that by the Sobolev interpolation the dissipation still can give some control on the energy: for $-s<\ell\le N$, by Lemma \ref{1-sinte}, we interpolate
to get
\begin{equation}\label{heat 3}
\norm{\nabla^\ell u(t)}_{L^2} \le \norm{\Lambda^{-s} u(t)}_{L^2}^\frac{1}{\ell+1+s}\norm{\nabla^{\ell+1} u(t)}_{L^2}^\frac{\ell+s}{\ell+1+s}.
\end{equation}
Combining \eqref{heat 3} and \eqref{heat 2} (with $\ell=-s$), we obtain
\begin{equation}\label{heat 4}
\norm{\nabla^{\ell+1} u(t)}_{L^2}\ge \norm{\Lambda^{-s} u_0}_{L^2}^{-\frac{1}{\ell+s}}\norm{\nabla^\ell u(t)}_{L^2}^{1+\frac{1}{\ell+s}}.
\end{equation}
Plugging \eqref{heat 4} into \eqref{heat 1}, we deduce that there exists a constant $C_0>0$ such
that
\begin{equation}\label{heat 5}
\frac{d}{dt}\norm{\nabla^\ell u}_{L^2}^2+C_0\left(\norm{\nabla^\ell u}_{L^2}^2\right)^{1+\frac{1}{\ell+s}}\le 0.
\end{equation}
Solving this inequality directly, and by \eqref{heat 2}, we obtain the following decay result:
\begin{equation}\label{heat 555}
\norm{\nabla^\ell u(t)}_{L^{2}}^{2}\leq \left( \norm{\nabla^\ell u_0}%
_{L^{2}}^{-\frac{2}{\ell +s}}+\frac{C_{0}t}{\ell +s}\right) ^{-(\ell
+s)}\leq C_{0}(1+t)^{-(\ell +s)}.
\end{equation}
Hence, we conclude our decay results by the pure energy method. Indifferent from the $L^p$--$L^2$ approach, an important feature here is that the $\dot{H}^{-s}$ norm of the solution is preserved along time evolution and this norm of initial data enhances the decay rate of the solution with the factor $s/2$. Although \eqref{heat 555} can be proved by the Fourier analysis or
spectral method, the same strategy in our proof can be applied to nonlinear
system with two essential points in the proof: (1) closing the energy
estimates at each $\ell $-th level (referring to the order of the spatial
derivatives of the solution); (2) deriving a novel negative Sobolev
estimates for nonlinear system which requires $s<3/2$ ($n/2$ for dimension $n$).

In the rest of this paper, except that we will collect in Appendix the analytic tools which will be used, we will apply the energy method illustrated above to prove Theorem \ref{theorem1} and Theorem \ref{theorem2} in Section \ref{sec 2} and Section \ref{sec 3} respetively. However, we will be
not able to close the energy estimates at each $\ell $-th level as the
heat equation. This is caused by the ``degenerate" dissipative structure of the linear homogenous system of \eqref{VPB_per} when using our energy method. More precisely, the linear energy identity of the problem reads as
\begin{equation}\label{energy identity boltzmann}
\frac{1}{2}\frac{d}{dt}\left(\norm{\left[\nabla^\ell f,\nabla^\ell g\right]}^2+\norm{\nabla^\ell\nabla_x\Phi}^2\right)+(\mathcal{L}_1\nabla^\ell f,\nabla^\ell f)+(\mathcal{L}_2\nabla^\ell g,\nabla^\ell g)=0.
\end{equation}
It is well-known that $\mathcal{L}_1$ and $\mathcal{L}_2$ are only positive with respect to the microscopic part respectively, that is, there exists a $\sigma_0>0$ such that
\begin{equation}\label{coercive boltzmann}
(\mathcal{L}_1\nabla^\ell f,\nabla^\ell f)\ge \sigma_0 \norm{\nabla^\ell\{{\bf I-P_1}\}f}_\nu^2\ \text{ and }\ (\mathcal{L}_2\nabla^\ell g,\nabla^\ell g)\ge \sigma_0 \norm{\nabla^\ell\{{\bf I-P_2}\}g}_\nu^2.
\end{equation}
To rediscover the dissipative estimate for the hydrodynamic part
$\left[{\bf P_1}f,{\bf P_2}g\right]$ and the electric field $\nabla_x\Phi$, we will use the linearized equation of \eqref{VPB_per} via constructing the interactive energy functional $G_\ell$ between $\nabla^\ell $ and $\nabla^{\ell+1} $ of the solution to deduce
\begin{equation}\label{3232}
\begin{split}
&\frac{dG_\ell}{dt}+\norm{\nabla^{\ell+1} \mathbf{P_1}f }_{L^2}^2+\norm{\nabla^{\ell} \mathbf{P_2}g }_{L^2}^2+\norm{\nabla^{\ell+1} \mathbf{P_2}g }_{L^2}^2+\norm{\left[\nabla^{\ell} \nabla_x\Phi, \nabla^{\ell+1} \nabla_x\Phi,\nabla^{\ell+2} \nabla_x\Phi\right]}_{L^2}^2
\\&\quad\lesssim  \norm{\left[\nabla^{\ell}\{
\mathbf{I-P_1}\} f,\nabla^{\ell}\{
\mathbf{I-P_2}\} g  \right]}_{L^2}^2+\norm{\left[\nabla^{\ell+1}\{
\mathbf{I-P_1}\} f,\nabla^{\ell+1}\{
\mathbf{I-P_2}\} g  \right]}_{L^2}^2.
\end{split}
\end{equation}
This implies that to get the dissipative estimate at each level for the missing part in the standard energy inequality it requires us to do the energy estimates \eqref{energy identity boltzmann} at two levels (referring to the order of the spatial derivatives of the solution). To get around this obstacle, the idea is to construct some energy functionals $\widetilde{\mathcal{E}}_{\ell }(t)$, $0\leq \ell \leq N-1$ (less than $N-1$ is restricted by \eqref{3232}),
\begin{equation}
\widetilde{\mathcal{E}}_{\ell }(t)\backsim \sum_{\ell \leq k\leq N}\norm{\nabla^k f(t)}_{L^2}%
^{2},  \notag
\end{equation}%
which has a \textit{minimum} derivative count of \ $\ell ,$ and we will derive some Lyapunov-type inequalities for these energy functionals in which the corresponding dissipation (denoted by $\widetilde{\mathcal{D}}_\ell(t)$) can be related to the energy similarly as \eqref{heat 4} by the Sobolev interpolation. This can be easily established for the linear homogeneous problem along our analysis, however, for the nonlinear problem \eqref{VPB_per}, it is much more complicated
due to the nonlinear estimates. This is the second point of this paper as in \cite{GW} that we will use extensively the Sobolev interpolation of the the Gagliardo-Nirenberg inequality between high-order and low-order spatial derivatives to expect to bound the nonlinear terms by $\sqrt{\mathcal{E}_N(t)}\widetilde{\mathcal{D}}_\ell(t)$ that can be absorbed. But this can not be achieved well at this moment and we will be left with two extra terms: one term is related to a sum of norms of $\nabla_x\Phi$ and the other term is related to a sum of velocity-weighted norms of $[f,g]$, as stated in \eqref{energy estimate 4}. Note that when taking $\ell=0,1$ in \eqref{energy estimate 4}, we can absorb these two unpleasant terms and then we get Theorem \ref{theorem1} after we complete the negative Sobolev estimates. While for $\ell\ge 2$, we need to assume the weighted norm of the initial data. With the help of these weighted norms, we will succeed in removing the sum of velocity-weighted norms of $[f,g]$ from the right hand side of \eqref{energy estimate 4} to get \eqref{energy estimate 4'} in which we can take $\ell=2$. On the other hand, we can revisit the equations $\eqref{VPB_per}_2$--$\eqref{VPB_per}_3$ to deduce a further energy estimates \eqref{micro g} for $g$ and $\nabla_x\Phi$ which kills the sum of norms of $\nabla_x\Phi$ in the right hand side of \eqref{energy estimate 4'} to get \eqref{energy estimate 411'} in which we can take $\ell=3,\dots,N-1$. This energy estimates \eqref{micro g} also implies the exponential decay of $g$ and $\nabla_x\Phi$. Finally, to estimate the negative Sobolev norm in Lemma \ref{lemma H-s}, we need to restrict that $s<3/2$ when estimating $\Lambda^{-s}$ acting on the nonlinear terms by using the Hardy-Littlewood-Sobolev inequality and also we need to separate the cases that $s\in[0,1/2]$ and $s\in(1/2,3/2)$. We remark that it is also
important that we use the Minkowski's integral inequality to interchange the
order of integrations in $v$ and $x$ in order to estimate the nonlinear
terms and that we extensively use the splitting $f=\mathbf{P_1}f+{\{\mathbf{I-P_1
}}\}f$ and $g=\mathbf{P_2}g+{\{\mathbf{I-P_2
}}\}g$. Once these estimates are
obtained, Theorem \ref{theorem1} and Theorem \ref{theorem2} follow by the interpolation between negative and positive Sobolev norms
similarly as the heat equation case.

To end this introduction, we want to emphasize that the results of the two-species Vlasov-Poisson-Boltzmann system in Theorem \ref{theorem1} and \ref{theorem2} reveal the essential difference when compared to the one-species Vlasov-Poisson-Boltzmann system or the two-species Vlasov-Boltzmann system. For the one-species Vlasov-Poisson-Boltzmann system \cite{DS1}, the electric field decays at the optimal algebraic rate; For the two-species Vlasov-Boltzmann system \cite{W} the disparity between two species of particles decays at the optimal algebraic rate as the Boltzmann equation. Besides this, our achievement heavily relies on the special cancelation property between two species that our reformulation \eqref{VPB_per} displays well which gives the dissipative estimates of the $L^2$ norm of electric field. This cancelation phenomenon was also observed in \cite{S2006} for the study of the two-species Vlasov-Maxwell-Boltzmann system. The natural generalization of this paper is to using our energy method to revisit the decay rate of the two-species Vlasov-Maxwell-Boltzmann system \cite{W2}. Another interesting application is to revisit the one-species Vlasov-Poisson-Boltzmann system system in which the $L^2$ norm of electric field is not included in the dissipation, and we expect to investigate it in the future study.

\section{Energy estimates and proof of Theorem \ref{theorem1}}\label{sec 2}

\subsection{Basic energy estimates}\label{sec 2.1}
In this subsection,  we will derive the a priori nonlinear energy estimates
for the system \eqref{VPB_per}.
We shall first establish the energy estimates on the pure spatial derivatives of solutions. First of all, notice that the dissipation estimate in
\eqref{coercive boltzmann} is degenerate, and it only controls the microscopic part $\left[\{{\bf I- P_1}\}f,\{{\bf I- P_2}\}g\right]$. Hence, to get the full dissipation estimate we shall estimate the hydrodynamic part $[{\bf P_1}f,{\bf P_2}g]$ and the electric field $\nabla_x\Phi$ in
terms of the microscopic part, up to some (small) error terms. This will be done in the following lemma.

\begin{lemma}\label{lemma positive}
If $\mathcal{E}_N(t)\le \delta$, then for $k=0,\dots, N-1$,  we have
\begin{equation}\label{f positive estimate}
\begin{split}
\frac{dG_f^k}{dt}+C\norm{\nabla^{k+1}  {\bf P }_1f}^2 &\lesssim \norm{\nabla^{k}\{{\bf I-P_1}\} f }^2+\norm{\nabla^{k+1}\{{\bf I-P_1}\} f
}^2\\&\quad+\delta^2\left(\norm{\nabla^{k+1}g}^2+\norm{\nabla^{k+1} \nabla_x\Phi}^2\right)
\end{split}
\end{equation}
and
\begin{equation}\label{g positive estimate}
\begin{split}
&\frac{d}{dt}G_g^k(t)+C\left(\norm{\nabla^{k}  {\bf P }_2g }^2+\norm{\nabla^{k+1}{\bf P }_2g }^2+\norm{\left[\nabla^k \nabla_x
\Phi,\nabla^{k+1}\nabla_x \Phi,\nabla^{k+2}\nabla_x \Phi\right]}^2\right)
\\&\quad\lesssim\norm{\nabla^{k}\{{\bf I-P_2}\} g}^2+\norm{\nabla^{k+1}\{{\bf I-P_2}\} g}^2+\delta^2 \norm{\nabla^{k+1}f}^2.
\end{split}
\end{equation}
Here $G_f^k(t)$ and $G_g^k(t)$ are defined by \eqref{G_f^k(t)} and \eqref{G_g^k(t)} respectively which satisfy the estimates
\begin{equation}\label{G_fg^k estimate}
\left|G_f^k\right|\lesssim\norm{\nabla^k f}^2+\norm{\nabla^{k+1} f}^2\ \text{ and }\ \left|G_g^k\right|\lesssim\norm{\nabla^k g}^2+\norm{\nabla^{k+1}
g}^2+\norm{\nabla^k \nabla_x\Phi}^2.
\end{equation}
\end{lemma}

\begin{proof}
As in \cite{G2006,J2009,DY,DS1,DS2}, the proof is based on the careful analysis of the {\it local conservation laws} and the {\it macroscopic
equations} which are derived from the so called macro-micro decomposition. Our contribution is to carefully estimate the nonlinear terms so that we
may close the energy estimates at each $\ell$-level in our weaker sense.

 First, multiplying $\eqref{VPB_per}_1$ by the collision invariants
$1,v,|v|^2/2$, $\eqref{VPB_per}_2$ by $1$, and then integrating in $v\in \r3$, we get the local conservation laws
\begin{equation}\label{local conservation laws 0}
\begin{split}
&\partial_t\int_{\r3} F\,dv + \nabla_x\cdot\int_{\r3} vF\,dv=0,
\\&\partial_t\int_{\r3} vF\,dv + \nabla_x\cdot\int_{\r3} v\otimes vF\,dv-\nabla_x\Phi\int_{\r3} G\,dv=0,
\\&\partial_t\int_{\r3} \frac{1}{2}|v|^2F\,dv + \nabla_x\cdot\int_{\r3} \frac{1}{2}|v|^2v F\,dv-\nabla_x\Phi\cdot \int_{\r3} v G\,dv=0,
\\ &\partial_t\int_{\r3} G\,dv + \nabla_x\cdot\int_{\r3} vG\,dv=0.
\end{split}
\end{equation}
Plugging $F=\mu+\sqrt{\mu}\left({\bf P_1}f+{\{\bf I- P_1\}}f\right)$ and $G=\sqrt{\mu}\left({\bf P_2}g+\{{\bf I- P_2}\}g\right)$ into \eqref{local
conservation laws 0}, and using the representations \eqref{hydrodynamic field 1} and \eqref{hydrodynamic field 2}, we obtain
\begin{equation}\label{local conservation laws}
\begin{split}
&\partial_t a + \nabla_x\cdot b=0,
\\&\partial_t b + \nabla_x\cdot(a+4c)+\nabla_x\cdot \mathcal{A}(\{{\bf I-P_1}\}f)=\nabla_x\Phi\, d,
\\&\partial_t (3a+12c)+\nabla_x\cdot(5 b)+\nabla_x\cdot\mathcal{B}(\{{\bf I-P_1}\} f) =\nabla_x\Phi\cdot \mathcal{D}( \{{\bf
I-P_2}\}g),
\\ &\partial_t d + \nabla_x\cdot \mathcal{D}( \{{\bf
I-P_2}\}g)=0,
\end{split}
\end{equation}
where we have defined the moment functions $\mathcal{A}=(\mathcal{A}_{ij})_{3\times 3}$, $\mathcal{B}=(\mathcal{B}_1,\mathcal{B}_2,\mathcal{B}_3)$
and $\mathcal{D}=(\mathcal{D}_1,\mathcal{D}_2,\mathcal{D}_3)$ by
\begin{equation}\label{moment functions}
 \mathcal{A}_{ij}(h)=\langle h,v_iv_j\sqrt{\mu}\rangle,\quad \mathcal{B}_{i}(h)=\langle h,|v|^2v_i\sqrt{\mu}\rangle\ \text{ and }\ \mathcal{D}_{i}(h)=\langle h,v_i\sqrt{\mu}\rangle.
\end{equation}
Notice that $\eqref{local conservation laws}_1$ and $\eqref{local conservation laws}_3$ implies
\begin{equation}\label{local_c}
\partial_t c+\frac{1}{6}\nabla_x\cdot b+ \frac{1}{12}\nabla_x\cdot \mathcal{B}(\{{\bf I-P_1}\} f)=\frac{1}{12}\nabla_x\Phi\cdot \mathcal{D}( \{{\bf
I-P_2}\}g).
\end{equation}
Also, the Poisson equation $\eqref{VPB_per}_3$ reads as
\begin{equation}\label{Poisson}
\Delta_x\Phi=d.
\end{equation}

Next, plugging $f={\bf P_1}f+{\{\bf I- P_1\}}f$ and $g={\bf P_2}g+\{{\bf I- P_2}\}g$ into \eqref{VPB_per}, we obtain
\begin{equation}\label{P equation}
\begin{split}
&\partial_t {\bf P_1} f + v\cdot\nabla_x {\bf P_1} f=-\partial_t{\{\bf I- P_1\}} f+\mathfrak{L}_1+\mathfrak{N}_1,
\\ &\partial_t {\bf P_2}g + v\cdot\nabla_x {\bf P_2} g-\nabla_x\Phi\cdot v\sqrt{\mu} =-\partial_t{\{\bf I- P_2\}} g+\mathfrak{L}_2+\mathfrak{N}_2,
\end{split}
\end{equation}
where the linear terms $\mathfrak{L}_1$ and $\mathfrak{L}_2$ are denoted by
\begin{equation}
\mathfrak{L}_1=-v\cdot\nabla_x {\{\bf I- P_1\}} f-\mathcal{L}_1 {\{\bf I- P_1\}} f,\quad \mathfrak{L}_2=-v\cdot\nabla_x {\{\bf I- P_2\}}
g-\mathcal{L}_2 {\{\bf I- P_2\}} g.
\end{equation}
 Motivated by \cite{DS1,DS2}, we will not use all of the {\it macroscopic equations} but only use the evolution of some
moments of the microscopic part that appeared in \eqref{local conservation laws}. To simplify the notations, we define two more moment functions
$\widetilde{\mathcal{A}}=(\widetilde{\mathcal{A}}_{ij})_{3\times 3}$ and
$\widetilde{\mathcal{B}}=(\widetilde{\mathcal{B}}_1,\widetilde{\mathcal{B}}_2,\widetilde{\mathcal{B}}_3)$ by
\begin{equation}\label{two moment functions}
 \widetilde{\mathcal{A}}_{ij}(h)=\mathcal{A}_{ij}(h)-\langle h,\sqrt{\mu}\rangle,
 \quad \widetilde{\mathcal{B}}_{i}(h)=\frac{1}{10}\left(\mathcal{B}_{i}(h)-5\mathcal{D}_{i}(h)\right).
\end{equation}
Applying $\widetilde{\mathcal{A}}_{ij}(\cdot), \widetilde{\mathcal{B}}_i(\cdot)$ to $\eqref{P equation}_1$ and $\mathcal{D}_i(\cdot)$ to $\eqref{P
equation}_2$ respectively, we get
\begin{equation}\label{moment equations}
\begin{split}
& \partial_t[\widetilde{\mathcal{A}}_{ij}({\{\bf I- P_1\}}
f)+4c\delta_{ij}]+\partial_ib_j+\partial_jb_i=\widetilde{\mathcal{A}}_{ij}(\mathfrak{L}_1+\mathfrak{N}_1),
\\&\partial_t \widetilde{\mathcal{B}}_{i}({\{\bf I- P_1\}} f)+2\partial_i c=\widetilde{\mathcal{B}}_{i}(\mathfrak{L}_1+\mathfrak{N}_1),
\\&\partial_t\mathcal{D}_{i}({\{\bf I- P_2\}} g)+\partial_i d-\partial_i \Phi=\mathcal{D}_{i}(\mathfrak{L}_2+\mathfrak{N}_2).
 \end{split}
\end{equation}
Notice that for fixed $i$, one can deduce from $\eqref{moment equations}_1$ that
\begin{equation}\label{moment b equation}
\begin{split}
&-\partial_t\left[\sum_j\partial_j\widetilde{\mathcal{A}}_{ji}({\{\bf I- P_1\}} f)+\frac{1}{2}\partial_i\widetilde{\mathcal{A}}_{ii}({\{\bf I- P_1\}}
f)\right]-\Delta_x b_i-\partial_{ii}b_i
\\&\quad=\frac{1}{2}\sum_{j\neq i}\partial_i\widetilde{\mathcal{A}}_{jj}(\mathfrak{L}_1+\mathfrak{N}_1)-\sum_{j}\partial_j\widetilde{\mathcal{A}}_{ji}(\mathfrak{L}_1+\mathfrak{N}_1).
\end{split}
\end{equation}

Bypassing the nonlinear coupling terms, the local conservation laws $\eqref{local conservation laws}_1$--$\eqref{local conservation laws}_3$,
$\eqref{local_c}$ and the macroscopic equations $\eqref{moment equations}_1$--$\eqref{moment equations}_2$, \eqref{moment b equation} on the
macroscopic coefficients $a,b,c$ are decoupled from those on $d$ and $\nabla_x\Phi$, and they have the same structure as the pure Boltzmann case. So
by these equations and following the proof in \cite{G2006} or \cite{DY,DS1,DS2,J2009}, we can deduce the dissipative estimates of $ a, b, c$: for
$|\gamma|=k$ with $k=0,\dots,N-1$,
\begin{eqnarray}
&&\label{b_estimate}
\begin{split}
&\frac{d}{dt}\left( \partial^\gamma\sum_j\partial_j \widetilde{\mathcal{A}}_{ji}({\{\bf I- P_1\}} f)+\frac{1}{2}\partial^\gamma\partial_i\widetilde{\mathcal{A}}_{ii}({\{\bf I- P_1\}} f),
\partial^\gamma  b_i\right)+\frac{1}{2}\norm{\partial^\gamma  \nabla_x b}^2
\\&\quad\le\varepsilon
\left(\norm{\partial^\gamma \nabla_x a}^2+\norm{\partial^\gamma \nabla_x c}^2+\norm{\partial^\gamma \left(\nabla_x\Phi\, d\right)}^2\right)
\\&\qquad
 + C_\varepsilon \left(\norm{\partial^\gamma {\{\bf I- P_1\}} f}^2+\norm{\nabla_x \partial^\gamma{\{\bf I- P_1\}} f}^2+\norm{
\partial^\gamma\mathfrak{N}_{1,\parallel}}^2\right),
\end{split}
\\&&\label{c_estimate}
\begin{split}
&\frac{d}{dt}\left(
\partial^\gamma \widetilde{\mathcal{B}}_{i}({\{\bf I- P_1\}} f), \partial^\gamma \partial_i  c\right)+ \norm{\partial^\gamma  \nabla_x c}^2
\\&\quad\le \varepsilon \left(\norm{\partial^\gamma \nabla_x b}^2+\norm{\partial^\gamma \left(\nabla_x\Phi\cdot \mathcal{D}( \{{\bf I-P_2}\}g)\right)}^2\right)
\\&\qquad + C_\varepsilon \left(\norm{\partial^\gamma {\{\bf I- P_1\}} f}^2+\norm{\nabla_x \partial^\gamma{\{\bf I- P_1\}} f}^2+\norm{
\partial^\gamma\mathfrak{N}_{1,\parallel}}^2\right),
\end{split}
\\&&\label{a_estimate}
\begin{split}
&\frac{d}{dt}(\partial^\gamma b,\partial^\gamma\nabla_xa)+\frac{1}{2}\norm{\partial^\gamma\nabla_xa}^2
 \\&\quad \le C\left(\norm{\partial^\gamma\nabla_x
b}^2+\norm{\partial^\gamma\nabla_xc}^2+\norm{\nabla_x \partial^\gamma{\{\bf I- P_1\}} f}^2+\norm{
\partial^\gamma(\nabla_x\Phi\, d)}^2\right).
\end{split}
\end{eqnarray}
Here $\mathfrak{N}_{1,\parallel}$ is defined as $\langle\mathfrak{N}_{1},\zeta\rangle$ with $\zeta$ is some linear combination of
$[\sqrt{\mu},v_i\sqrt{\mu},v_iv_j\sqrt{\mu},v_i|v|^2\sqrt{\mu}]$, etc. We multiply \eqref{b_estimate} and \eqref{c_estimate} by a constant $M>0$ and
then sum up them as well as \eqref{a_estimate}. We first fix $M$ to be sufficiently large such that the first two terms in the right hand side of
\eqref{a_estimate} can be absorbed, and then for such fixed $M$ we further let $\varepsilon>0$ sufficiently small such that the first term in the
right hand side of \eqref{c_estimate} and the first two terms in the right hand side of \eqref{b_estimate} can be absorbed. Hence, we obtain
\begin{equation}\label{full_estimate 1}
\begin{split}
\frac{d}{dt}G_f^k(t)+C\norm{\nabla^{k+1}{\bf P }_1f}^2 &\lesssim \norm{\nabla^{k}\{{\bf I-P_1}\} f }^2+\norm{\nabla^{k+1}\{{\bf I-P_1}\} f }^2+\norm{
\nabla^{k}\mathfrak{N}_{1,\parallel}}^2
\\&\quad + \norm{
\nabla^{k}(\nabla_x\Phi\, d)}^2+\norm{\nabla^{k}\left(\nabla_x\Phi\cdot \mathcal{D}( \{{\bf I-P_2}\}g)\right)}^2.
\end{split}
\end{equation}
where $G_f^k(t)$ is defined by
\begin{equation}\label{G_f^k(t)}
\begin{split}
&G_f^k(t):=\left.\sum_{|\gamma|=k}\right\{M\left(\partial^\gamma\sum_j\partial_j \widetilde{\mathcal{A}}_{ji}({\{\bf I- P_1\}}
f)+\frac{1}{2}\partial^\gamma\partial_i\widetilde{\mathcal{A}}_{ii}({\{\bf I- P_1\}} f),
\partial^\gamma b_i\right)
\\&\quad  \qquad\qquad\quad\ \left.+M\left(
\partial^\gamma \widetilde{\mathcal{B}}_{i}({\{\bf I- P_1\}} f), \partial^\gamma \partial_i  c\right)
+(\partial^\gamma b,\partial^\gamma\nabla_xa)\right\}.
\end{split}
\end{equation}

We now focus on the derivation of the dissipation estimates of $d$ and $\nabla_x\Phi$. However, we shall estimate in the same spirit with a bit more
attention on the electric field. Applying $\partial^\gamma$ with $|\gamma|=k$ to $\eqref{moment equations}_3$, multiplying the resulting equations by
$\partial^\gamma (\partial_i d-\partial_i\Phi)$ and then integrating by parts over $x\in \r3$, by the Poisson equation \eqref{Poisson}, we get
\begin{equation}\label{d_es_0}
\begin{split}
&\norm{\partial^\gamma  \nabla_x d}^2+\norm{\partial^\gamma \nabla_x \Phi}^2 \\&\quad=2(\partial^\gamma\partial_i \Phi,\partial^\gamma
\partial_i d)-(
\partial_t\partial^\gamma \mathcal{D}_{i}({\{\bf I- P_2\}} g), \partial^\gamma(\partial_i d-\partial_i\Phi))
\\&\qquad+ ( \partial^\gamma \mathcal{D}_{i}(\mathfrak{L}_2+\mathfrak{N}_2), \partial^\gamma(\partial_i d-\partial_i\Phi))
\\&\quad=-2\norm{\partial^\gamma  d}^2-\frac{d}{dt}(
\partial^\gamma \mathcal{D}_{i}({\{\bf I- P_2\}} g), \partial^\gamma(\partial_i d-\partial_i\Phi))
\\&\qquad-(\partial^\gamma \partial_i \mathcal{D}_{i}({\{\bf I- P_2\}} g), \partial^\gamma \partial_td)
-(\partial^\gamma  \mathcal{D}_{i}({\{\bf I- P_2\}} g), \partial^\gamma \partial_t \partial_i\Phi)
\\&\qquad+ ( \partial^\gamma \mathcal{D}_{i}(\mathfrak{L}_2+\mathfrak{N}_2), \partial^\gamma(\partial_i d-\partial_i\Phi))
\\&\quad\le-2\norm{\partial^\gamma  d}^2-\frac{d}{dt}\left(
\partial^\gamma \mathcal{D}_{i}({\{\bf I- P_2\}} g), \partial^\gamma(\partial_i d-\partial_i\Phi)\right)
\\&\qquad+\varepsilon\left(\norm{\partial^\gamma \partial_td}^2+\norm{\partial^\gamma \partial_t\nabla_x\Phi}^2 +\norm{\partial^\gamma\nabla_x d}^2+\norm{\partial^\gamma\nabla_x \Phi}^2\right)
\\&\qquad + C_\varepsilon \left(\norm{\partial^\gamma {\{\bf
I- P_2\}} g}^2+\norm{\nabla_x \partial^\gamma{\{\bf I- P_2\}} g}^2+\norm{ \partial^\gamma\mathfrak{N}_{2,\parallel}}^2\right) .
\end{split}
\end{equation}
It is crucial that it follows from $\eqref{local conservation laws}_4$ and \eqref{Poisson} that the follows hold:
\begin{equation}\label{d_01}
\norm{\partial^\gamma \partial_td}^2\le C\norm{\nabla_x \partial^\gamma{\{\bf I- P_2\}} g}^2
\end{equation}
and
\begin{equation}\label{d_02}
\norm{\partial^\gamma \partial_t\nabla_x\Phi}^2 \le C\norm{\partial^\gamma{\{\bf I- P_2\}} g}^2.
\end{equation}
Hence, by \eqref{d_01}--\eqref{d_02} and the Poisson equation \eqref{Poisson} again, we deduce from \eqref{d_es_0} that, by taking $\varepsilon>0$
 sufficiently small,
\begin{equation}\label{full_estimate 2}
\begin{split}
&\frac{d}{dt}G_g^k(t)+C\left(\norm{\nabla^{k}  {\bf P }_2g }^2+\norm{\nabla^{k+1}{\bf P }_2g }^2+\norm{[\nabla^k \nabla_x \Phi,\nabla^{k+1}\nabla_x
\Phi,\nabla^{k+2}\nabla_x \Phi]}^2\right)
\\&\quad\lesssim\norm{\nabla^{k}\{{\bf I-P_2}\} g }^2+\norm{\nabla^{k+1}\{{\bf I-P_2}\} g}^2
+\norm{\nabla^{k}\mathfrak{N}_{2,\parallel}}^2.
\end{split}
\end{equation}
where $G_g^k(t)$ is defined by
\begin{equation}\label{G_g^k(t)}
\begin{split}
&G_g^k(t):= \sum_{|\gamma|=k} \left(
\partial^\gamma \mathcal{D}_{i}({\{\bf I- P_2\}} g), \partial^\gamma(\partial_i d-\partial_i\Phi)\right).
\end{split}
\end{equation}

It is clear that $G_f^k(t)$ and $G_g^k(t)$ satisfy the estimates \eqref{G_fg^k estimate}, and it remains to estimate the nonlinear terms in the right
hand side of \eqref{full_estimate 1} and \eqref{full_estimate 2}. The main idea is that we will carefully adjust the index of spatial derivatives when estimating the nonlinear terms so that they
can be bounded by the right hand side of \eqref{f positive estimate} or \eqref{g positive estimate}. We begin with the estimate of the term $\norm{
\nabla^k\mathfrak{N}_{2,\parallel}}^2:=\sum_{|\gamma|=k}\norm{\langle\partial^\gamma \mathfrak{N}_2,\zeta\rangle}^2$. First, by the estimate
\eqref{nonlineares2} of Lemma \ref{nonlinearcol1} and the fact that $\zeta$ decays exponentially in $v$, we have
\begin{equation}\label{positive f}
\begin{split}
\sum_{|\gamma|=k}\norm{\langle\partial^\gamma\Gamma(g,f),\zeta\rangle}^2&\lesssim\sum_{|\gamma|=k}\sum_{\gamma_1\le \gamma}\left\|
\langle\Gamma(\partial^{\gamma_1} g,\partial^{\gamma-\gamma_1}f),\zeta\rangle \right\|^2\\&\lesssim \sum_{|\gamma_1|\le k}\norm{|\nabla^{|\gamma_1|}
g|_2|\nabla^{k-|\gamma_1|}f|_2}^2.
\end{split}
\end{equation}
By H\"older's inequality, Minkowski's integral inequality \eqref{min es} of Lemma \ref{Minkowski}, the Sobolev interpolation of Lemma
\ref{interpolation} and Young's inequality, we obtain
\begin{equation}\label{N2 0}
\begin{split}
\norm{|\nabla^{|\gamma_1|} g|_2|\nabla^{k-|\gamma_1|}f|_2} &\lesssim \norm{\nabla^{|\gamma_1|} g}_{L^6_xL^2_v}\norm{
\nabla^{k-|\gamma_1|}f}_{L^3_xL^2_v}\lesssim \norm{\nabla^{|\gamma_1|} g}_{L^2_vL^6_x}\norm{ \nabla^{k-|\gamma_1|}f}_{L^2_vL^3_x}
\\&\lesssim \norm{g}^{1-\frac{|\gamma_1|+1}{k+1}}\norm{\nabla^{k+1}g}^{\frac{|\gamma_1|+1}{k+1}}
\norm{\nabla^\alpha f}^{\frac{|\gamma_1|+1}{k+1}}\norm{\nabla^{k+1}f}^{1-\frac{|\gamma_1|+1}{k+1}}
\\&\lesssim \delta\left(\norm{\nabla^{k+1}f}+\norm{\nabla^{k+1}g}\right).
\end{split}
\end{equation}
Here $\alpha$ comes from the adjustment of the index and is defined by
\begin{equation}
\begin{split}
&\frac{1}{3}-\frac{k-|\gamma_1|}{3}=\left(\frac{1}{2}-\frac{\alpha}{3}\right)\times
\frac{|\gamma_1|+1}{k+1}+\left(\frac{1}{2}-\frac{k+1}{3}\right)\times\left(1-\frac{|\gamma_1|+1}{k+1}\right)
\\&\quad\Longrightarrow\alpha=\frac{k+1}{2(|\gamma_1|+1)}\le \frac{k+1}{2}\le \frac{N}{2}.
\end{split}
\end{equation}
Hence, we have
\begin{equation}\label{N2 1}
\sum_{|\gamma|=k}\norm{\langle\partial^\gamma\Gamma(g,f),\zeta\rangle}^2 \lesssim \delta^2\left(\norm{\nabla^{k+1}f}^2+\norm{\nabla^{k+1}g}^2\right).
\end{equation}
Similarly, we may apply the same arguments to obtain, by Lemma \ref{1interpolation} and Lemma \ref{interpolation},
\begin{equation}\label{N2 2}
\begin{split}
\sum_{|\gamma|=k}\norm{\langle\partial^\gamma (\nabla_x\Phi\cdot vf),\zeta\rangle}^2 &\lesssim \sum_{|\gamma|=k}\sum_{\gamma_1\le
\gamma}\norm{\partial^{\gamma_1}\langle f,v\zeta\rangle\cdot \partial^{\gamma-\gamma_1} \nabla_x\Phi}^2
\\&\lesssim \delta^2\left(\norm{\nabla^{k+1}\langle f,v\zeta\rangle}^2+\norm{\nabla^{k+1} \nabla_x\Phi}^2\right)
\\&\lesssim \delta^2\left(\norm{\nabla^{k+1}f}^2+\norm{\nabla^{k+1} \nabla_x\Phi}^2\right),
\end{split}
\end{equation}
and with the additional integration by parts over $v$-variable to have
\begin{equation}\label{N2 3}
\begin{split}
\sum_{|\gamma|=k}\norm{\langle\partial^\gamma (\nabla_x\Phi\cdot\nabla_vf),\zeta\rangle}^2 &=\sum_{|\gamma|=k}\norm{\langle\partial^\gamma
(\nabla_x\Phi f),\nabla_v\zeta\rangle}^2
\\&\lesssim \sum_{|\gamma|=k}\sum_{\gamma_1\le \gamma}\norm{\partial^{\gamma_1}\langle f,\nabla_v\zeta\rangle\cdot \partial^{\gamma-\gamma_1} \nabla_x\Phi}^2
\\&\lesssim \delta^2\left(\norm{\nabla^{k+1}\langle f,\nabla_v\zeta\rangle}^2+\norm{\nabla^{k+1} \nabla_x\Phi}^2\right)
\\&\lesssim \delta^2\left(\norm{\nabla^{k+1}f}^2+\norm{\nabla^{k+1} \nabla_x\Phi}^2\right).
\end{split}
\end{equation}
Thus, summing up the estimates \eqref{N2 1}--\eqref{N2 3}, we have
\begin{equation}\label{N2}
\norm{\nabla^{k}\mathfrak{N}_{2,\parallel}}^2\lesssim \delta^2\left(\norm{\left[\nabla^{k+1}f,\nabla^{k+1}g\right]}^2+\norm{\nabla^{k+1}
\nabla_x\Phi}^2\right).
\end{equation}
Then plugging the estimate \eqref{N2} into \eqref{full_estimate 2}, since $\delta$ is small, we obtain \eqref{g positive estimate}.

Similarly, we have
\begin{equation}\label{N1}
\norm{\nabla^{k}\mathfrak{N}_{1,\parallel}}^2\lesssim \delta^2\left(\norm{\left[\nabla^{k+1}f,\nabla^{k+1}g\right]}^2+\norm{\nabla^{k+1}
\nabla_x\Phi}^2\right).
\end{equation}
However, we will also use the same argument above to estimate the remaining two terms in the right hand side of \eqref{full_estimate 1} to obtain
\begin{equation}\label{positive 1}
\begin{split}
\sum_{|\gamma|=k}\norm{\partial^\gamma(d\nabla_x\Phi)}^2&\lesssim
 \delta^2\left(\norm{\nabla^{k+1}d}^2+ \norm{\nabla^{k+1}\nabla_x\Phi}^2\right)
 \\&\lesssim
 \delta^2\left(\norm{\nabla^{k+1}g}^2+ \norm{\nabla^{k+1}\nabla_x\Phi}^2\right)
 \end{split}
\end{equation}
and
\begin{equation}\label{positive 2}
\begin{split}
\sum_{|\gamma|=k}\norm{\partial^\gamma \left(\nabla_x\Phi\cdot \mathcal{D}( \{{\bf I-P_2}\}g)\right)}^2 &\lesssim
\delta^2\left(\norm{\nabla^{k+1}\mathcal{D}( \{{\bf I-P_2}\}g)}^2+ \norm{\nabla^{k+1}\nabla_x\Phi}^2\right)
\\&\lesssim \delta^2\left(\norm{\nabla^{k+1} g}^2+ \norm{\nabla^{k+1}\nabla_x\Phi}^2\right).
\end{split}
\end{equation}
Then plugging the estimates \eqref{N1}--\eqref{positive 2} into \eqref{full_estimate 1}, since $\delta$ is small, we obtain \eqref{f positive
estimate}.
\end{proof}

Now we consider the energy estimates for the pure spatial derivatives of the solution.
\begin{lemma}\label{lemma spatial energy}
If $\mathcal{E}_N(t)\le \delta$, then for $k=0,\dots, N-1$, we have
\begin{equation}\label{spatial energy estimate 1}
\begin{split}
&\frac{d}{dt}\left(\norm{\left[\nabla^k f,\nabla^k g\right]}^2+\norm{\nabla^k \nabla_x\Phi}^2\right) +C\norm{\left[\nabla^k\{{\bf I-P_1}\}
f,\nabla^k\{{\bf I-P_2}\} g\right]}_\nu^2
\\&\quad \lesssim \delta \left(\norm{\left[\nabla^{k+1} f,\nabla^{k+1} g\right]}^2+\norm{\nabla^{k}{\bf P_2}g}^2
+\norm{\left[\nabla^k\nabla_x\Phi,\nabla^{k+1}\nabla_x\Phi\right]}^2\right.
\\&\qquad\quad\ \left.+\sum_{2\le \ell \le
N}\norm{\nabla^\ell\nabla_x\Phi}^2+\sum_{1\le \ell \le N}\norm{\left[\nabla^{\ell}\{{\bf I- P_1}\}f,\nabla^{\ell}\{{\bf I- P_2}\}g\right]}_\nu^2\right);
\end{split}
\end{equation}
and for $k=N$, we have
\begin{equation}\label{spatial energy estimate 2}
\begin{split}
&\frac{d}{dt}\left(\norm{\left[\nabla^N f,\nabla^N g\right]}^2+\norm{\nabla^N \nabla_x\Phi}^2\right) +C\norm{\left[\nabla^N\{{\bf I-P_1}\}
f,\nabla^N\{{\bf I-P_2}\} g\right]}_\nu^2
\\&\quad\lesssim\delta \left(\norm{\left[\nabla^N f,\nabla^N g\right]}^2+\sum_{2\le \ell \le N}\norm{\nabla^\ell\nabla_x\Phi}^2+\sum_{1\le \ell \le N}\norm{\left[\nabla^{\ell}\{{\bf I- P_1}\}f,\nabla^{\ell}\{{\bf I- P_2}\}g\right]}_\nu^2\right).
\end{split}
\end{equation}
\end{lemma}
\begin{proof}
The standard $\partial^\gamma$ with $|\gamma|=k$ energy estimates on $\eqref{VPB_per}_1$ and $\eqref{VPB_per}_2$ give rise to
\begin{equation}\label{EL11 f}
\frac{1}{2}\frac{d}{dt}\norm{\partial^\gamma f}^2+ (\mathcal{L}_1\partial^\gamma f ,\partial^\gamma f ) =(\partial^\gamma \mathfrak{N}_1,
\partial^\gamma f):=I_1+I_2+I_3
\end{equation}
and
\begin{equation}\label{EL11 g}
\frac{1}{2}\frac{d}{dt}\norm{\partial^\gamma g}^2-(\partial^\gamma\nabla_x\Phi\cdot v\sqrt{\mu},
\partial^\gamma g)+ (\mathcal{L}_2\partial^\gamma g ,\partial^\gamma g )
=(\partial^\gamma \mathfrak{N}_2,
\partial^\gamma g):=J_1+J_2+J_3.
\end{equation}

We first estimate the left hand side of \eqref{EL11 f}--\eqref{EL11 g}. By $\eqref{local conservation laws}_4$ and \eqref{Poisson}, we have
\begin{equation}\label{poisson estimate}
\begin{split}
-(\partial^\gamma\nabla_x\Phi\cdot v\sqrt{\mu}, \partial^\gamma g)&=(\partial^\gamma\Phi, \nabla_x\cdot  \mathcal{D}( \{{\bf I-P_2}\}g))
\\&=-(\partial^\gamma\Phi, \partial_t\partial^\gamma d)=-(\partial^\gamma\Phi,
\partial_t\partial^\gamma\Delta_x\Phi)=\frac{1}{2}\frac{d}{dt}\norm{\partial^\gamma\nabla_x\Phi}^2.
\end{split}
\end{equation}
The estimate \eqref{3} of Lemma \ref{linear c} implies
\begin{equation}\label{micro est 1}
(\mathcal{L}_1\partial^\gamma f ,\partial^\gamma f ) \ge \sigma_0\norm{\partial^\gamma\{{\bf I-P_1}\} f}_\nu^2
\end{equation}
and
\begin{equation}\label{micro est 2}
(\mathcal{L}_2\partial^\gamma g ,\partial^\gamma g ) \ge \sigma_0\norm{\partial^\gamma\{{\bf I-P_2}\} g}_\nu^2.
\end{equation}

Next, we turn to estimate the right hand side of \eqref{EL11 f}--\eqref{EL11 g}. First, by the collision invariant property and the estimate
\eqref{nonlineares3} (with $\eta=1/2$) of Lemma \ref{nonlinearcol1}, we obtain
\begin{equation}\label{line 1 g}
\begin{split}
J_1&:=(\partial^\gamma\Gamma(g,f),\partial^\gamma g)=\sum_{\gamma_1\le \gamma}C_\gamma^{\gamma_1}(\Gamma(\partial^{\gamma_1} g,
\partial^{\gamma-\gamma_1} f),\partial^\gamma \{{\bf I-P_2}\}g)
\\&\lesssim\sum_{\gamma_1\le \gamma}\norm{\nu^{-1/2}\Gamma(\partial^{\gamma_1} g,
\partial^{\gamma-\gamma_1} f)}\norm{\nu^{1/2}\partial^\gamma \{{\bf I-P_2}\}g}
\\&\lesssim\sum_{\gamma_1\le \gamma}\norm{|\nabla^{|\gamma_1|} g|_2|\nabla^{k-|\gamma_1|} f|_\nu
+|\nabla^{|\gamma_1|} g|_\nu|\nabla^{k-|\gamma_1|} f|_2}\norm{\nabla^k\{{\bf I- P_2}\}g}_\nu.
\end{split}
\end{equation}
For the first term in the right hand side of \eqref{line 1 g}, we use the splitting $f={\bf P_1}f+\{{\bf I-P_1}\}f$ to have
\begin{equation}\label{444}
\begin{split}
\norm{|\nabla^{|\gamma_1|} g|_2|\nabla^{k-|\gamma_1|} f|_\nu }&\lesssim\norm{|\nabla^{|\gamma_1|} g|_2|\nabla^{k-|\gamma_1|} f|_2 }
+\norm{|\nabla^{|\gamma_1|} g|_2|\nabla^{k-|\gamma_1|}\{{\bf I-P_1}\} f|_\nu} \\&:= J_{11}+J_{12}.
\end{split}
\end{equation}
For the term $J_{11}$, if $k=0,\dots,N-1$, it has been already bounded in \eqref{N2 0} as
\begin{equation}\label{N2 0''}
J_{11}\lesssim \delta\left(\norm{\nabla^{k+1}f}+\norm{\nabla^{k+1}g}\right);
\end{equation}
while for $k=N$, we then take the supremum in $x$ on the factor with the less number of derivatives to have, for instance if $|\gamma_1|\le
\frac{N}{2}$, by Lemma \ref{Minkowski} and Lemma \ref{interpolation}
\begin{equation}\label{temple}
\begin{split}
J_{11}&\lesssim\norm{\nabla^{|\gamma_1|}g}_{L^\infty_xL^2_v}\norm{\nabla^{N-|\gamma_1|}f}\lesssim\norm{\nabla^{|\gamma_1|}g}_{L^2_vL^\infty_x}\norm{\nabla^{N-|\gamma_1|}f}
\\ &\lesssim \norm{\nabla^{\alpha}g}^{1-\frac{|\gamma_1|}{N}}\norm{\nabla^{N}g}^{\frac{|\gamma_1|}{N}}\norm{f}^{\frac{|\gamma_1|}{N}}\norm{ \nabla^Nf}^{1-\frac{|\gamma_1|}{N}}
\\&\lesssim \delta \left(\norm{\nabla^{N}f}+\norm{\nabla^{N}g}\right),
\end{split}
\end{equation}
where we have denoted $\alpha$ by
\begin{equation}
\begin{split}
&-\frac{|\gamma_1|}{3}=\left(\frac{1}{2}-\frac{\alpha}{3}\right)\times\left(1-\frac{|\gamma_1|}{N}\right)+\left(\frac{1}{2}-\frac{N}{3}\right)\times\frac{|\gamma_1|}{N}
\\&\quad\Longrightarrow\alpha=\frac{3N}{2(N-|\gamma_1|)}\le 3\ \text{ since }|\gamma_1|\le N/2.
\end{split}
\end{equation}
For the term $J_{12}$, note that we can only bound the $\nu$-weighted factor by the dissipation, so we can not pursue as before to adjust the index.
Noticing that $\{{\bf I-P_1}\}f$ is always part of the dissipation. If $|\gamma_1|\le k-2$ (if $k-2<0$, then it's nothing in this case, etc.) and
hence $k-|\gamma_1|\ge 2$, then we bound
\begin{equation}
\begin{split}
J_{12}&\le \norm{\nabla^{|\gamma_1|}g}_{L^\infty_x}\norm{\nabla^{k-|\gamma_1|}\{{\bf I-P_1}\}f}_\nu
\\&\lesssim \delta \sum_{2\le \ell\le N}\norm{\nabla^\ell \{{\bf I-P_1}\}f}_\nu;
\end{split}
\end{equation}
and if $|\gamma_1|\ge k-1$ and hence $k-|\gamma_1|\le 1$, then by Sobolev's inequality, we bound
\begin{equation}
\begin{split}
J_{12}&\le \norm{\nabla^{|\gamma_1|}g}\norm{|\nabla^{k-|\gamma_1|}\{{\bf I-P_1}\}f|_\nu}_{L^\infty_x} \le
\norm{\nabla^{|\gamma_1|}g}\norm{\nu^{1/2}\nabla^{k-|\gamma_1|}\{{\bf I-P_1}\}f}_{L^2_vL^\infty_x}
\\&\lesssim \delta \sum_{1\le \ell\le 3}\norm{\nabla^\ell \{{\bf I-P_1}\}f}_\nu^2.
\end{split}
\end{equation}
Hence, we have
\begin{equation}\label{j12}
J_{12}\lesssim \delta\sum_{1\le \ell\le N}\norm{\nabla^\ell \{{\bf I-P_1}\}f}_\nu^2.
\end{equation}
Thus, we complete the estimate of the first term in the right hand side of \eqref{line 1 g}. Applying the same argument (exchange $f$ and $g$) to the
other term, we conclude that for $k=0,\dots,N-1$,
\begin{equation}\label{g non 11}
J_1\lesssim\delta\left(\norm{\left[\nabla^{k+1}f,\nabla^{k+1}g\right]}+\norm{\nabla^k\{{\bf I- P_2}\}g}_\nu+\sum_{1\le \ell\le
N}\norm{\left[\nabla^\ell \{{\bf I-P_1}\}f,\nabla^\ell \{{\bf I-P_2}\}g\right]}_\nu^2\right);
\end{equation}
and for $k=N$,
\begin{equation}\label{g non 12}
J_1\lesssim\delta\left(\norm{\left[\nabla^Nf,\nabla^Ng\right]}+\sum_{1\le \ell\le N}\norm{\left[\nabla^\ell \{{\bf I-P_1}\}f,\nabla^\ell \{{\bf
I-P_2}\}g\right]}_\nu^2\right).
\end{equation}
Similarly, we have that for $k=0,\dots,N-1$,
\begin{equation}\label{f non 11}
I_1:=(\partial^\gamma\Gamma(f,f),\partial^\gamma f)\lesssim\delta\left(\norm{\nabla^{k+1}f}+\norm{\nabla^k\{{\bf I- P_1}\}f}_\nu+\sum_{1\le \ell\le
N}\norm{\nabla^\ell \{{\bf I-P_1}\}f}_\nu^2\right);
\end{equation}
and for $k=N$,
\begin{equation}\label{f non 12}
I_1\lesssim\delta\left(\norm{\nabla^Nf}+\sum_{1\le \ell\le N}\norm{\nabla^\ell \{{\bf I-P_1}\}f}_\nu^2\right).
\end{equation}

Next, for the second term $J_2$ in the right hand side of \eqref{EL11 g}, we first have
\begin{equation}
\begin{split}
J_2&:=(\partial^\gamma(\nabla_x\Phi\cdot vf), \partial^\gamma g)=\sum_{\gamma_1\le\gamma}C_\gamma^{\gamma_1}(\partial^{\gamma_1}\nabla_x\Phi\cdot v
\partial^{\gamma-\gamma_1} f, \partial^\gamma g)\\&\lesssim\sum_{\gamma_1\le \gamma}\norm{|\nabla^{|\gamma_1|} \nabla_x\Phi||\nabla^{k-|\gamma_1|} f|_\nu
}\norm{\nabla^k g}_\nu.
\end{split}
\end{equation}
Applying the same arguments that for \eqref{444}, we obtain that for $k=0,\dots,N-1$,
\begin{equation}\label{g non 21}
J_2\lesssim\delta\left(\norm{\nabla^{k+1}f}+\norm{\nabla^{k+1}\nabla_x\Phi}+\norm{\nabla^k g}_\nu+\sum_{1\le \ell\le N}\norm{\nabla^\ell \{{\bf
I-P_1}\}f}_\nu^2\right);
\end{equation}
and for $k=N$,
\begin{equation}\label{g non 22}
J_2\lesssim\delta\left(\norm{\nabla^Nf}+\norm{\nabla^N\nabla_x\Phi}+\norm{\nabla^N g}_\nu+\sum_{1\le \ell\le N}\norm{\nabla^\ell \{{\bf
I-P_1}\}f}_\nu^2\right).
\end{equation}

For the second term $I_2$ in the right hand side of \eqref{EL11 f}, we again use the splitting to have
\begin{equation}
\begin{split}
I_2&:=(\partial^\gamma(\nabla_x\Phi\cdot vg), \partial^\gamma f)=\sum_{\gamma_1\le\gamma}C_\gamma^{\gamma_1}(\partial^{\gamma_1}\nabla_x\Phi\cdot v
\partial^{\gamma-\gamma_1} g, \partial^\gamma f)\\&=\sum_{\gamma_1\le\gamma}C_\gamma^{\gamma_1}\left\{(\partial^{\gamma_1}\nabla_x\Phi\cdot v
\partial^{\gamma-\gamma_1} g, \partial^\gamma {\bf P_1}f)+(\partial^{\gamma_1}\nabla_x\Phi\cdot v
\partial^{\gamma-\gamma_1}g, \partial^\gamma  \{{\bf I-P_1}\}f)\right\}
\\&:=I_{21}+I_{22}.
\end{split}
\end{equation}
Notice that the term $I_{22}$ can be bounded as that for \eqref{444} and we have that for $k=0,\dots,N-1$,
\begin{equation}\label{es i 221}
I_{22}\lesssim\delta\left(\norm{\nabla^{k+1}g}+\norm{\nabla^{k+1}\nabla_x\Phi}+\norm{\nabla^k \{{\bf I-P_1}\}f}_\nu+\sum_{1\le \ell\le
N}\norm{\nabla^\ell \{{\bf I-P_2}\}g}_\nu^2\right);
\end{equation}
and for $k=N$,
\begin{equation}\label{es i 222}
I_{22}\lesssim\delta\left(\norm{\nabla^Ng}+\norm{\nabla^N\nabla_x\Phi}+\norm{\nabla^N \{{\bf I-P_1}\}f}_\nu+\sum_{1\le \ell\le N}\norm{\nabla^\ell
\{{\bf I-P_2}\}g}_\nu^2\right).
\end{equation}
While for the first term $I_{21}$, we have
\begin{equation}\label{I_def}
I_{21}\lesssim \int_{\r3}|\nabla^{|\gamma_1|}\nabla_x\Phi| \left| \nabla^{k-|\gamma_1|}g\right|_2|\nabla^k f|_2\,dx.
\end{equation}
For $0\le k\le N-1$, by H\"older's inequality, Minkowski's integral inequality \eqref{min es} of Lemma \ref{Minkowski} and  Sobolev's inequality, we
have
\begin{equation}
\begin{split}
I_{21}&\lesssim\norm{\nabla^{|\gamma_1|}\nabla_x\Phi}_{L^3_x}\norm{\nabla^{k-|\gamma_1|}g}\norm{\nabla^k f}_{L^6_xL^2_v} \lesssim
\norm{\nabla^{|\gamma_1|}\nabla_x\Phi}_{L^3_x}\norm{\nabla^{k-|\gamma_1|}g}\norm{\nabla^k f}_{L^2_vL^6_x}
\\&\lesssim \norm{\nabla^{|\gamma_1|}\nabla_x\Phi}_{L^3_x}\norm{\nabla^{k-|\gamma_1|}g}\norm{\nabla^{k+1} f}
\end{split}
\end{equation}
If $\gamma_1=\gamma$, then in this case we have
\begin{equation}
\begin{split}
I_{21}\lesssim\norm{\nabla^k\nabla_x\Phi}_{L^3_x}\norm{g}\norm{\nabla^{k+1} f}\lesssim
\delta\left(\norm{\nabla^k\nabla_x\Phi}^2+\norm{\nabla^{k+1}\nabla_x\Phi}^2+\norm{\nabla^{k+1} f}^2\right).
\end{split}
\end{equation}
If $|\gamma_1|\le k-1$ (it is nothing if $k-1<0$), then in this case by Lemma \ref{1interpolation} and Lemma \ref{interpolation}, we have
\begin{equation}
\begin{split}
I_{21}&\lesssim
\norm{\nabla^{\alpha}\nabla_x\Phi}^{1-\frac{|\gamma_1|}{k}}\norm{\nabla^{k}\nabla_x\Phi}^{\frac{|\gamma_1|}{k}}\norm{g}^{\frac{|\gamma_1|}{k}}\norm{
\nabla^kg}^{1-\frac{|\gamma_1|}{k}}\norm{\nabla^{k+1} f}
\\&\lesssim \delta\left(\norm{\nabla^k\nabla_x\Phi}^2+\norm{\nabla^{k}g}^2+\norm{\nabla^{k+1} f}^2\right),
\end{split}
\end{equation}
where we have denoted $\alpha$ by
\begin{equation}
\begin{split}
&\frac{1}{3}-\frac{|\gamma_1|}{3}=\left(\frac{1}{2}-\frac{\alpha}{3}\right)\times\left(1-\frac{|\gamma_1|}{k}\right)+\left(\frac{1}{2}-\frac{k}{3}\right)\times\frac{|\gamma_1|}{k}
\\&\quad\Longrightarrow\alpha=\frac{k}{2(k-|\gamma_1|)}\le \frac{k}{2}\ \text{ since }|\gamma_1|\le k-1.
\end{split}
\end{equation}
Hence for $0\le k\le N-1$, we obtain
\begin{equation}\label{term11}
\begin{split}
I_{21}\lesssim \delta\left(\norm{\nabla^k\nabla_x\Phi}^2+\norm{\nabla^{k+1}\nabla_x\Phi}^2+\norm{\nabla^{k+1} f}^2+\norm{\nabla^{k}g}^2\right).
\end{split}
\end{equation}
Now for $k=N$, we may use the arguments as in \eqref{temple} to have
\begin{equation}\label{term12}
I_{21}\lesssim \delta\left(\norm{ \nabla^{N}g }^2+\norm{\nabla^N\nabla_x\Phi}^2+\norm{\nabla^{N}f}^2\right).
\end{equation}
Hence, in light of \eqref{es i 221}--\eqref{es i 222} and \eqref{term11}--\eqref{term12}, we may conclude that for $k=0,\dots,N-1$,
\begin{equation}\label{f non 21}
\begin{split}
I_{2}&\lesssim\delta\left(\norm{\left[\nabla^{k+1}f,\nabla^{k+1}g\right]}+\norm{\left[\nabla^k\nabla_x\Phi,\nabla^{k+1}\nabla_x\Phi\right]}^2+\norm{\nabla^{k}g}^2\right.
\\&\left.\qquad\ +\norm{\nabla^k \{{\bf I-P_1}\}f}_\nu+\sum_{1\le \ell\le N}\norm{\nabla^\ell \{{\bf I-P_2}\}g}_\nu^2\right);
\end{split}
\end{equation}
and for $k=N$,
\begin{equation}\label{f non 22}
I_{2}\lesssim\delta\left(\norm{\left[\nabla^{N}f,\nabla^{N}g\right]}^2+\norm{\nabla^N\nabla_x\Phi}+\norm{\nabla^N \{{\bf I-P_1}\}f}_\nu+\sum_{1\le
\ell\le N}\norm{\nabla^\ell \{{\bf I-P_2}\}g}_\nu^2\right).
\end{equation}

Now we turn to the third term $J_3$ in the right hand side of \eqref{EL11 g}. There is one worst case involving ($k+1$)-th derivative. We should
estimate this $(k+1)$-th order derivative term together with the similar term stemming from the third term $I_3$ in \eqref{EL11 f} to be canceled by
the integration by parts over $v$-variable. Hence, we estimate them together to have
\begin{equation}\label{line2}
\begin{split}
I_3+J_3&:=(\partial^\gamma(\nabla_x\Phi\cdot\nabla_vf), \partial^\gamma g)+(\partial^\gamma(\nabla_x\Phi\cdot\nabla_vg), \partial^\gamma f)
\\&=(\nabla_x\Phi\cdot\nabla_v\partial^\gamma f, \partial^\gamma g)+(\nabla_x\Phi\cdot\nabla_v \partial^\gamma g, \partial^\gamma f)
\\&\quad+\sum_{0\neq\gamma_1\le\gamma}C_\gamma^{\gamma_1}\left\{(\partial^{\gamma_1}\nabla_x\Phi\cdot\nabla_v \partial^{\gamma-\gamma_1} f, \partial^\gamma g)+(\partial^{\gamma_1}\nabla_x\Phi\cdot\nabla_v \partial^{\gamma-\gamma_1} g, \partial^\gamma f)
\right\}
\\&=\sum_{0\neq\gamma_1\le\gamma}C_\gamma^{\gamma_1}\left\{(\partial^{\gamma_1}\nabla_x\Phi\cdot\nabla_v \partial^{\gamma-\gamma_1} f, \partial^\gamma {\bf P_2}g)
+(\partial^{\gamma_1}\nabla_x\Phi\cdot\nabla_v \partial^{\gamma-\gamma_1} f, \partial^\gamma \{{\bf I-P_2}\}g) \right.
\\&\qquad\qquad\qquad\ \left.+(\partial^{\gamma_1}\nabla_x\Phi\cdot\nabla_v \partial^{\gamma-\gamma_1} g, \partial^\gamma {\bf P_1} f)
+(\partial^{\gamma_1}\nabla_x\Phi\cdot\nabla_v \partial^{\gamma-\gamma_1} g, \partial^\gamma \{{\bf I- P_1}\} f)\right\}.
\end{split}
\end{equation}
For the third term in the right hand side of \eqref{line2}, we integrate by parts in $v$ and use the fast that ${\bf P_1}f$ decay exponentially in
$v$ to have
\begin{equation}
\begin{split}
(\partial^{\gamma_1}\nabla_x\Phi\cdot\nabla_v \partial^{\gamma-\gamma_1} g, \partial^\gamma {\bf P_1} f)&=-(\partial^{\gamma_1}\nabla_x\Phi
\partial^{\gamma-\gamma_1} g, \nabla_v\partial^\gamma {\bf P_1}f)
\\&\lesssim \int_{\r3}|\nabla^{|\gamma_1|}\nabla_x\Phi| \left| \nabla^{k-|\gamma_1|}g\right|_2|\nabla^k f|_2\,dx.
\end{split}
\end{equation}
This term appeared in \eqref{I_def} that has already been bounded. Similarly for the first term,
\begin{equation}
\begin{split}
(\partial^{\gamma_1}\nabla_x\Phi\cdot\nabla_v \partial^{\gamma-\gamma_1} f, \partial^\gamma {\bf P_2}g)&=-(\partial^{\gamma_1}\nabla_x\Phi
\partial^{\gamma-\gamma_1} f, \nabla_v\partial^\gamma {\bf P_2}g)\\&\lesssim \int_{\r3}|\nabla^{|\gamma_1|}\nabla_x\Phi| \left| \nabla^{k-|\gamma_1|}f\right|_2|\nabla^k g|_2\,dx,
\end{split}
\end{equation}
This term can be bounded similarly as the first term in \eqref{444}. While for the remaining two terms in \eqref{line2}, note that we can only bound
the factors involving the velocity derivative by the energy (not dissipation; otherwise, our method would not work.), so we can not pursue as before to adjust the
index. Thus, we may bound, for instance the third term, if $1\le|\gamma_1|< N/2$,
\begin{equation}
\begin{split}
(\partial^{\gamma_1}\nabla_x\Phi\cdot\nabla_v \partial^{\gamma-\gamma_1} f, \partial^\gamma \{{\bf I- P_2}\} g) &\lesssim \norm{\nabla^{|\gamma_1|}
\nabla_x\Phi}_{L^\infty_x}\norm{\nabla_v \nabla^{k-|\gamma_1|} f}\norm{\nabla^k\{{\bf I- P_2}\}g}\\&\lesssim \delta\left(\sum_{2\le \ell \le
[\frac{N}{2}]+2}\norm{\nabla^\ell\nabla_x\Phi}^2+\norm{\nabla^k \{{\bf I-P_2}\}g}^2\right);
\end{split}
\end{equation}
and if $|\gamma_1|\ge N/2$
\begin{equation}
\begin{split}
(\partial^{\gamma_1}\nabla_x\Phi\cdot\nabla_v \partial^{\gamma-\gamma_1} f, \partial^\gamma \{{\bf I- P_2}\} g) &\lesssim
\norm{\nabla^{|\gamma_1|}\nabla_x\Phi}\norm{\nabla_v \nabla^{k-|\gamma_1|} f}_{L^\infty_xL^2_v}\norm{\nabla^k\{{\bf I- P_2}\}g}
\\&\lesssim
\delta\left( \sum_{[\frac{N}{2}]\le\ell\le N}\norm{\nabla^\ell\nabla_x\Phi}^2+\norm{\nabla^k \{{\bf I-P_2}\}g}^2\right);
\end{split}
\end{equation}
Hence, we may conclude that for $0\le k\le N-1$,
\begin{equation}\label{fg non 31}
\begin{split}
I_3+J_3&\lesssim \delta\left(\norm{\nabla^{k+1} f}^2+\norm{\nabla^{k}g}^2+\norm{\nabla^k\nabla_x\Phi}^2+\norm{\nabla^{k+1}\nabla_x\Phi}^2\right.
\\&\qquad\ \ \left.+\sum_{2\le
\ell \le N}\norm{\nabla^\ell\nabla_x\Phi}^2+\norm{\left[\nabla^k \{{\bf I-P_1}\}f,\nabla^k \{{\bf I-P_2}\}g\right]}^2\right);
\end{split}
\end{equation}
and for $k= N$,
\begin{equation}\label{fg non 32}
\begin{split}
&(\partial^\gamma(\nabla_x\Phi\cdot\nabla_vg), \partial^\gamma f)+(\partial^\gamma(\nabla_x\Phi\cdot\nabla_vf), \partial^\gamma g)
\\&\quad\lesssim\delta\left(\norm{\left[\nabla^Nf,\nabla^N g\right]}^2+\sum_{2\le \ell \le
N}\norm{\nabla^\ell\nabla_x\Phi}^2+\norm{\left[\nabla^N \{{\bf I-P_1}\}f,\nabla^N \{{\bf I-P_2}\}g\right]}_\nu^2\right).
\end{split}
\end{equation}

 Consequently, plugging the estimates \eqref{poisson estimate}--\eqref{micro est 2} into the left hand side of \eqref{EL11 f}--\eqref{EL11 g}
 and then bounding the right hand side by \eqref{g non 11}, \eqref{f non 11}, \eqref{g non 21}, \eqref{f non 21} and \eqref{fg non 31}
 for $0\le k\le N-1$, and bounding it by \eqref{g non 12}, \eqref{f non 12}, \eqref{g non 22}, \eqref{f non 22} and \eqref{fg non 32}
 for $k=N$, since $\delta$ is small, we obtain \eqref{spatial energy estimate 1} and \eqref{spatial energy estimate 2} respectively.
\end{proof}

Until to now, we may conclude our energy estimates of the pure spatial derivatives as follows. Let $0\le\ell\le N-1$ and assume the a priori
estimates that $\mathcal{E}_N(t)\le \delta$ for sufficiently small $\delta>0$. Then summing up the estimates \eqref{spatial energy estimate 1} of Lemma
\ref{lemma spatial energy} from $k=\ell$ to $N-1$ and  adding the resulting estimate with the estimate \eqref{spatial energy estimate 2}, by changing
the index, we obtain
\begin{equation}\label{energy estimate 1}
\begin{split}
&\frac{d}{dt}\sum_{\ell\le k \le N}\left(\norm{\left[\nabla^k f,\nabla^k g\right]}^2+\norm{\nabla^k \nabla_x\Phi}^2\right) +C_1\sum_{\ell\le k \le
N}\norm{\left[\nabla^k\{{\bf I-P_1}\}  f,\nabla^k\{{\bf I-P_2}\}  g\right]}_\nu^2  \\&\quad \le C_2\delta \left(\sum_{\ell+1\le k \le
N}\norm{\left[\nabla^{k} f,\nabla^{k} g\right]}^2+\sum_{\ell\le k \le N-1}\norm{\nabla^{k}{\bf P_2}g}^2+\sum_{\ell\le k \le
N}\norm{\nabla^k\nabla_x\Phi}^2\right.
\\&\qquad\qquad\ \ \left.+\sum_{2\le k \le N}\norm{\nabla^k\nabla_x\Phi}^2+\sum_{1\le k \le N}\norm{\left[\nabla^{k}\{{\bf I-
P_1}\}f,\nabla^{k}\{{\bf I- P_2}\}g\right]}_\nu^2\right).
\end{split}
\end{equation}
On the other hand, we sum up the estimates \eqref{f positive estimate}--\eqref{g positive estimate} of Lemma \ref{lemma positive} from $k=\ell$ to
$N-1$, by changing the index, to obtain, since $\delta$ is small,
\begin{equation}\label{energy estimate 2}
\begin{split}
&\frac{d}{dt}\sum_{\ell\le k\le N-1}\left(G^k_f+G^k_g\right) +C_3 \left(\sum_{\ell+1\le k\le N} \norm{\nabla^{k}  {\bf P_1 }f }^2+\sum_{\ell\le k\le
N} \norm{\nabla^{k} {\bf P_2 }g}^2 +\sum_{\ell\le k\le N+1} \|\nabla^{k} \nabla_x\Phi \|^2\right)
\\&\quad\le C_4\sum_{\ell\le k\le N}\norm{\left[\nabla^{k}\{{\bf I-P _1}\} f,\nabla^{k}\{{\bf I-P _2}\} g\right] }_\nu^2.
\end{split}
\end{equation}
Then, multiplying \eqref{energy estimate 2} by a small number $\beta>0$ and then adding the resulting inequality with \eqref{energy estimate 1}, we
obtain
\begin{equation}\label{energy estimate 3}
\begin{split}
&\frac{d}{dt}\left(\sum_{\ell\le k \le N}\left(\norm{\left[\nabla^k f,\nabla^k g\right]}^2+\norm{\nabla^k \nabla_x\Phi}^2\right)+\beta\sum_{\ell\le
k\le N-1}\left(G^k_f+G^k_g\right)\right)
\\&\quad+(C_1-C_4\beta)\sum_{\ell\le k\le N}\norm{\left[\nabla^{k}\{{\bf I-P _1}\} f,\nabla^{k}\{{\bf I-P _2}\} g\right] }_\nu^2
\\&\quad+C_3\beta \left(\sum_{\ell+1\le k\le N} \norm{\nabla^{k}  {\bf P_1 }f}^2+\sum_{\ell\le k\le N} \norm{\nabla^{k} {\bf
P_2 }g}^2 +\sum_{\ell\le k\le N+1} \norm{\nabla^{k} \nabla_x\Phi }^2\right)
\\&\quad \le C_2\delta \left(\sum_{\ell+1\le k \le N}\norm{\left[\nabla^{k}
f,\nabla^{k} g\right]}^2+\sum_{\ell\le k \le N-1}\norm{\nabla^{k}{\bf P_2}g}^2+\sum_{\ell\le k \le N}\norm{\nabla^k\nabla_x\Phi}^2\right.
\\&\qquad\qquad\ \ \left.+\sum_{2\le k \le N}\norm{\nabla^k\nabla_x\Phi}^2+\sum_{1\le k \le N}\norm{\left[\nabla^{k}\{{\bf I-
P_1}\}f,\nabla^{k}\{{\bf I- P_2}\}g\right]}_\nu^2\right).
\end{split}
\end{equation}

We define $\widetilde{\mathcal{E}}_\ell(t)$ to be  the expression under the time derivative in \eqref{energy estimate 3}. We may now fix $\beta$ to
be sufficiently small so that $(C_1-2C_4\beta)>0$ and that $\widetilde{\mathcal{E}}_\ell(t)$ is equivalent to
\begin{equation}\nonumber
\sum_{\ell\le k \le N}\norm{\left[\nabla^k f,\nabla^k g\right]}^2+\sum_{\ell\le k \le N+1}\norm{\nabla^k \nabla_x\Phi}^2,
\end{equation}
where we have used the estimates \eqref{G_fg^k estimate} and the Poisson estimate. On the other hand, since $\beta$ is fixed and $\delta$ is small,
we can then absorb the first three terms in the right hand side of \eqref{energy estimate 3} to deduce that for $\ell=0,\dots,N-1$, by adjusting the
constant in the definition of $\widetilde{\mathcal{E}}_\ell(t)$,
\begin{equation}\label{energy estimate 4}
\begin{split}
&\frac{d}{dt}\widetilde{\mathcal{E}}_\ell+\sum_{\ell\le k\le N}\|[\nabla^{k}\{{\bf I-P _1}\} f,\nabla^{k}\{{\bf I-P _2}\} g] \|_\nu^2
\\&\quad+\sum_{\ell+1\le k\le N} \|\nabla^{k}  {\bf P_1 }f \|^2+\sum_{\ell\le k\le N} \norm{\nabla^{k} {\bf
P_2 }g}^2 +\sum_{\ell\le k\le N+1} \|\nabla^{k} \nabla_x\Phi \|^2
\\&\qquad\le C_5\delta \left(\sum_{2\le k \le N+1}\norm{\nabla^k\nabla_x\Phi}^2+\sum_{1\le k \le N}\norm{[\nabla^{k}\{{\bf I- P_1}\}f,\nabla^{k}\{{\bf I- P_1}\}g]}_\nu^2\right).
\end{split}
\end{equation}
\smallskip\smallskip

Now we turn to the energy estimates on the spatial-velocity mixed derivatives of the solution. First notice that for the hydrodynamic part $\left[{\bf P_1}f
,{\bf P_2}g \right]$,
\begin{equation}\label{hydro bound}
\norm{\left[\partial^\gamma_\beta {\bf P_1}f,\partial^\gamma_\beta {\bf P_2}g\right]}\lesssim\norm{\left[\partial^\gamma {\bf P_1}f,\partial^\gamma
{\bf P_2}g\right]}
\end{equation}
which has been estimated in \eqref{energy estimate 4}, it suffices to estimate the remaining microscopic part
\begin{equation}
\left[\partial^\gamma_\beta \{{\bf I-P_1}\}f,\partial^\gamma_\beta \{{\bf I-P_2}\}g\right]\nonumber
\end{equation}
for $|\gamma|+|\beta|\le N$ with $|\beta|\ge 1$ (and hence $|\gamma|\le N-1$).
\begin{lemma}\label{lemma spatial velocity}
If $\mathcal{E}_N(t)\le \delta$, then we have
\begin{equation}\label{I-P11}
\begin{split}
&\frac{d}{dt}\sum_{|\gamma|+|\beta|\le N\atop |\beta|\ge1}\norm{\left[\partial^\gamma_\beta\{{\bf I-P_1 }\}f,\partial^\gamma_\beta\{{\bf I-P_2
}\}g\right]}^2 + C\sum_{|\gamma|+|\beta|\le N\atop |\beta|\ge1}\norm{\left[\partial^\gamma_\beta\{{\bf I-P_1 }\}f,\partial^\gamma_\beta\{{\bf I-P_2
}\}g\right]}_\nu^2
\\&\quad\lesssim \sum_{1\le \ell \le N}\norm{\left[\nabla^\ell f,\nabla^\ell g\right]}^2+\sum_{0\le \ell \le N-1}\norm{\nabla^\ell \nabla_x\Phi}^2
\\&\qquad+\sum_{0\le\ell\le N}\norm{\left[\nabla^\ell \{{\bf I-P_1 }\}f,\nabla^\ell \{{\bf I-P_2 }\}g\right]}_\nu^2
+\sqrt{\mathcal{E}_N}\mathcal{D}_N.
\end{split}
\end{equation}
\end{lemma}
\begin{proof}
We take $\partial^\gamma_\beta$ (with $|\beta|=m\ge 1$) of the equations $\eqref{VPB_per}_1$--$\eqref{VPB_per}_2$ to get
\begin{eqnarray}\label{f 1}
&&\begin{split} &\partial_t\partial^\gamma_\beta\{{\bf I-P_1 }\}f + v\cdot\nabla_x\partial^\gamma_\beta\{{\bf I-P_1 }\} f  +
\partial^\gamma_\beta\mathcal{L}_1 \{{\bf  I-P_1 }\}f
\\&\qquad+\partial_t\partial^\gamma_\beta{\bf P_1 }f+v\cdot\nabla_x \partial^\gamma_\beta{\bf P_1}f +C_\beta^{\beta_1}\partial_{\beta_1}v
\cdot\nabla_x\partial^\gamma_{\beta-\beta_1} f
\\&\quad =  \partial^\gamma_\beta\left(\Gamma(f,f)+\frac{1}{2}\nabla_x\Phi\cdot vg-\nabla_x\Phi\cdot\nabla_vg \right),
 \end{split}
\\&&\label{g_1}
\begin{split}
&\partial_t\partial^\gamma_\beta\{{\bf I-P_2 }\}g + v\cdot\nabla_x\partial^\gamma_\beta\{{\bf I-P_2 }\} g  +  \partial^\gamma_\beta\mathcal{L}_2
\{{\bf  I-P_2 }\}g
\\&\qquad+\partial_t\partial^\gamma_\beta{\bf P_2 }g+v\cdot\nabla_x \partial^\gamma_\beta{\bf P_2}g +C_\beta^{\beta_1}\partial_{\beta_1}v
\cdot\nabla_x\partial^\gamma_{\beta-\beta_1} g-\partial^\gamma\nabla_x\Phi\cdot\partial_\beta( v\sqrt{\mu})
\\&\quad =  \partial^\gamma_\beta\left(\Gamma(g,f)+\frac{1}{2}\nabla_x\Phi\cdot vf-\nabla_x\Phi\cdot\nabla_vf \right).
 \end{split}
\end{eqnarray}
We illustrate only the estimate on $\{{\bf I-P_2 }\} g$ and the other one, $\{{\bf I-P_1 }\} f$, can be estimated in the same way. Taking the inner
product of \eqref{g_1} with $\partial^\gamma_\beta\{{\bf I-P_2 }\} g $, we obtain
\begin{equation}\label{I-P}
\begin{split}
&\frac{1}{2}\frac{d}{dt}\norm{\partial^\gamma_\beta\{{\bf I-P_2 }\}g}^2 + \left(\partial^\gamma_\beta\mathcal{L}_2 \{{\bf  I-P_2
}\}g,\partial^\gamma_\beta\{{\bf I-P_2 }\}g\right)
\\&\quad+\left(\partial_t\partial^\gamma_\beta{\bf P_2 }g+v\cdot\nabla_x \partial^\gamma_\beta{\bf P_2}g +C_\beta^{\beta_1}\partial_{\beta_1}v
\cdot\nabla_x\partial^\gamma_{\beta-\beta_1} g-\partial^\gamma\nabla_x\Phi\cdot\partial_\beta( v\sqrt{\mu}),\partial^\gamma_\beta\{{\bf I-P_2
}\}g\right)
\\&\quad =  \left(\partial^\gamma_\beta\left(\Gamma(g,f)+\frac{1}{2}\nabla_x\Phi\cdot vf-\nabla_x\Phi\cdot\nabla_vf\right),\partial^\gamma_\beta\{{\bf I-P_2 }\}g\right).
 \end{split}
\end{equation}

By the linear estimate \eqref{5} in Lemma \ref{linear c}, we know
\begin{equation}\label{vv1}
\left(\partial^\gamma_\beta\mathcal{L}_2 \{{\bf I-P_2}\}g ,\partial^\gamma_\beta \{{\bf I-P_2}\}g \right)\ge \frac{1}{2}\norm{\partial^\gamma_\beta
\{{\bf I-P_2}\}g}_\nu^2-C\norm{\partial^\gamma \{{\bf I-P_2}\}g}_\nu^2.
\end{equation}

We now estimate the second line in \eqref{I-P}. For the first two terms, by the local conservation laws $\eqref{local conservation laws}_4$ and Cauchy's inequality, we get
\begin{equation}\label{vv2}
\begin{split}
\left(\partial_t\partial^\gamma_\beta{\bf P_2 }g+v\cdot\nabla_x \partial^\gamma_\beta{\bf P_2}g ,\partial^\gamma_\beta\{{\bf I-P_2 }\}g\right) &
\lesssim\left(\norm{\partial_t\partial^\gamma d }+\norm{\nabla_x\partial^\gamma d }\right)\norm{\partial^\gamma_\beta \{{\bf I-P_2}\}g }
\\&  \le \frac{1}{12}\norm{\partial^\gamma_\beta \{{\bf
I-P_2}\}g }_\nu^2+C\norm{\nabla_x\partial^\gamma g }^2.
\end{split}
\end{equation}
For the third term, we use the splitting to have, since $|\beta-\beta_1|=m-1$,
\begin{equation}\label{vv3}
\begin{split}
&\left|\left(\partial^{\beta_1}v\cdot\nabla_x\partial^\gamma_{\beta-\beta_1}g ,\partial^\gamma_\beta \{{\bf I-P_2}\}g \right)\right|
\\&\quad\lesssim\left|\left(\partial^{\beta_1}v\cdot\nabla_x\partial^\gamma_{\beta-\beta_1}{\bf
P_2}g ,\partial^\gamma_\beta \{{\bf I-P_2}\}g \right)\right|+\left|\left(\partial^{\beta_1}v\cdot\nabla_x\partial^\gamma_{\beta-\beta_1}\{{\bf I-
P_2}\}g ,\partial^\gamma_\beta \{{\bf I-P_2}\}g \right)\right|
\\&\quad \le\frac{1}{12}\norm{\partial^\gamma_\beta \{{\bf
I-P_2}\}g }_\nu^2+ C\norm{\nabla_x\partial^\gamma g }^2 +C\sum_{|\gamma|+|\beta|\le N\atop|\beta|=m-1}\norm{\partial^\gamma_\beta \{{\bf I-P_2}\}g
}_\nu^2
\end{split}
\end{equation}
For the fourth term in the same line, we have
\begin{equation}\label{vv4}
\left|\left(\partial^\gamma\nabla_x\Phi\cdot\partial_\beta( v\sqrt{\mu}),\partial^\gamma_\beta\{{\bf I-P_2 }\}g\right)\right| \le
\frac{1}{12}\norm{\partial^\gamma_\beta\{{\bf I-P_2 }\}g}_\nu+C \norm{\partial^\gamma\nabla_x\Phi}.
\end{equation}

Now we turn to the third line in \eqref{I-P}. First, by Lemma 7 of \cite{S2006} (clearly, without taking the time derivatives), we have
\begin{equation}\label{vv5}
\left(\partial^\gamma_\beta\Gamma(g,f),\partial^\gamma_\beta\{{\bf I-P_2 }\}g\right)\lesssim\sqrt{\mathcal{E}_N}\mathcal{D}_N.
\end{equation}
Then it remains to estimate the last two terms related to the electric field. We do the splitting
\begin{equation}\label{electric field}
\begin{split}
&\left(\partial^\gamma_\beta\left(\frac{1}{2}\nabla_x\Phi\cdot vf-\nabla_x\Phi\cdot\nabla_vf\right),\partial^\gamma_\beta\{{\bf I-P_2 }\}g\right)
 \\&\quad=\left(\partial^\gamma_\beta\left(\frac{1}{2}\nabla_x\Phi\cdot v{\bf P_1}f-\nabla_x\Phi\cdot\nabla_v {\bf P_1}f\right),\partial^\gamma_\beta\{{\bf I-P_2 }\}g\right)
 \\&\qquad+\left(\partial^\gamma_\beta\left(\frac{1}{2}\nabla_x\Phi\cdot v \{{\bf I-P_1 }\} f-\nabla_x\Phi\cdot\nabla_v \{{\bf I-P_1 }\}f\right),\partial^\gamma_\beta\{{\bf I-P_2 }\}g\right).
\end{split}
\end{equation}
Since the hydrodynamic part is not affected by the velocity derivative and the $v$ factor as noted in \eqref{hydro bound}, hence the first term in
\eqref{electric field} can be bounded by $\sqrt{\mathcal{E}_N}\mathcal{D}_N$. The argument is similar as that in Lemma \ref{lemma spatial energy}, but
is a bit simpler since we do not adjust the index. The main concern is that there is one worst case involving ($N+1$)-th derivative in the second
term coming from the $\nabla_v$-Vlasov term. We should estimate this $(N+1)$-th order derivative term together with the similar term stemming from
the equation \eqref{f 1} to be canceled by the integration by parts:
\begin{equation}
\left(\partial^\gamma_\beta\left(\nabla_x\Phi\cdot\nabla_v \{{\bf I-P_1 }\}f\right),\partial^\gamma_\beta\{{\bf I-P_2 }\}g\right)
+\left(\partial^\gamma_\beta\left(\nabla_x\Phi\cdot\nabla_v \{{\bf I-P_1 }\}g\right),\partial^\gamma_\beta\{{\bf I-P_2 }\}f\right)=0.
\end{equation}
After this cancelation, we then can bound the second term in \eqref{electric field} by $\sqrt{\mathcal{E}_N}\mathcal{D}_N$. Plugging
these estimates and \eqref{vv1}--\eqref{vv5} into \eqref{I-P}, and doing the same estimates for \eqref{f 1} with respect to $\{{\bf I-P_1 }\} f$, summing up
$|\gamma|+|\beta|\le N$ with $|\beta|=m\ge1$, we obtain
\begin{equation}\label{I-Pm}
\begin{split}
&\frac{1}{2}\frac{d}{dt}\sum_{|\gamma|+|\beta|\le N\atop |\beta|=m}\norm{\left[\partial^\gamma_\beta\{{\bf I-P_1 }\}f,\partial^\gamma_\beta\{{\bf
I-P_2 }\}g\right]}^2 + \frac{1}{4}\sum_{|\gamma|+|\beta|\le N\atop |\beta|=m}\norm{\left[\partial^\gamma_\beta\{{\bf I-P_1
}\}f,\partial^\gamma_\beta\{{\bf I-P_2 }\}g\right]}_\nu^2
\\&\quad\lesssim \sum_{1\le \ell \le N-1}\norm{\left[\nabla^\ell f,\nabla^\ell g\right]}^2+\sum_{0\le \ell \le N-1}\norm{\nabla^\ell \nabla_x\Phi}^2+\sum_{0\le\ell\le N}\norm{\left[\nabla^\ell \{{\bf I-P_1 }\}f,\nabla^\ell \{{\bf I-P_2 }\}g\right]}_\nu^2
\\&\qquad+\sum_{|\gamma|+|\beta|\le N\atop |\beta|=m-1}\norm{\left[\partial^\gamma_\beta\{{\bf I-P_1 }\}f,\partial^\gamma_\beta\{{\bf I-P_2
}\}g\right]}_\nu^2+\sqrt{\mathcal{E}_N}\mathcal{D}_N.
\end{split}
\end{equation}
A simple recursive argument on \eqref{I-Pm} with the value of $m$ gives \eqref{I-P11}.\end{proof}

\subsection{Negative Sobolev estimates}\label{sec 2.2}
In this subsection, we will derive the evolution of the negative Sobolev norms of the solution. In order to estimate the nonlinear terms, we need to
restrict ourselves to $s\in (0,3/2)$. We will establish the following lemma.
\begin{lemma}\label{lemma H-s}
If $\mathcal{E}_N(t)\le \delta$, then for $s\in (0, 1/2]$, we have
\begin{equation}\label{H-s1}
\frac{d}{dt}\norm{\Lambda^{-s}f}^2+C\norm{\Lambda^{-s}\{{\bf I-P_1}\}f}_\nu^2\lesssim \left(\norm{\Lambda^{-s}f}+1\right)\mathcal{D}_N;
\end{equation}
and for $s\in (1/2, 3/2)$, we have
\begin{equation}\label{H-s2}
\frac{d}{dt}\norm{\Lambda^{-s}f}^2+C\norm{\Lambda^{-s}\{{\bf I-P}_1\}f}_\nu^2\lesssim \left(\norm{\Lambda^{-s}f}+1\right)\mathcal{D}_N+\norm{ f
}^{2s+1}\norm{ \nabla f   }^{3-2s}.
\end{equation}
\end{lemma}
\begin{proof}
Applying $\Lambda^{-s}$ to $\eqref{VPB_per}_1$ , and then taking the $L^2$ inner product with $\Lambda^{-s}f$, we have
\begin{equation}\label{Hs es 0}
\begin{split}
&\frac{1}{2}\frac{d}{dt}\norm{\Lambda^{-s}f}^2+ \sigma_0\norm{\Lambda^{-s}\{{\bf I-P_1}\} f}_\nu^2
\\&\quad \le \left(\Lambda^{-s}\Gamma(f,f),\Lambda^{-s}f\right)+\frac{1}{2}\left(\Lambda^{-s}\left(\nabla_x\Phi\cdot
vg\right), \Lambda^{-s} f\right) -\left(\Lambda^{-s}\left(\nabla_x\Phi\cdot\nabla_vg\right), \Lambda^{-s} f\right).
\end{split}
\end{equation}

We will estimate the right hand side of \eqref{Hs es 0} term by term. For the first term, by the collision invariant property, we have
\begin{equation}\label{Hs es 4}
\begin{split}
(\Lambda^{-s}\Gamma(f,f),\Lambda^{-s}f) &=\left(\Lambda^{-s}\Gamma(f,f), \Lambda^{-s}\{{\bf I-P}\}f\right)
\\&\le \norm{\Lambda^{-s}\left(\nu^{-\frac{1}{2}}\Gamma(f,f)\right)}\norm{\nu^{\frac{1}{2}}\Lambda^{-s}\{{\bf I-P}\}f}
\\&\le   C\norm{\Lambda^{-s}\left(\nu^{-\frac{1}{2}}\Gamma(f,f)\right)}^2+\frac{\sigma_0}{4}\norm{\Lambda^{-s}\{{\bf I-P}\}f}_\nu^2.
\end{split}
\end{equation}
To estimate the right hand side of \eqref{Hs es 4}, since $0<s<3/2$, we let $1<p<2$ to be with ${1}/{2}+{s}/{3}={1}/{p}$. By the estimate
\eqref{1Riesz es} of Riesz potential in Lemma \ref{1Riesz}, Minkowski's integral inequality \eqref{min es} of Lemma \ref{Minkowski}, and the estimate
\eqref{nonlineares3} (with $\eta=1/2$) of Lemma \ref{nonlinearcol1}, together with H\"older's inequality and the splitting $f={\bf P_1}f+\{{\bf I-P_1}\}f$, we obtain
\begin{equation}\label{Hs es 5}
\begin{split}
&\norm{\Lambda^{-s}\left(\nu^{-\frac{1}{2}}\Gamma(f,f)\right)} = \norm{\Lambda^{-s}\left(\nu^{-\frac{1}{2}}\Gamma(f,f)\right)}_{L_v^2L_x^2}^2
\\&\quad\lesssim \norm{\nu^{-\frac{1}{2}}\Gamma(f,f)}_{L_v^2L_x^{p}}
\le  \norm{\nu^{-\frac{1}{2}}\Gamma(f,f)}_{L_x^{p}L_v^2}
\\&\quad\lesssim \norm{|f|_2|f |_\nu}_{L_x^{p}}
\le  \norm{f}_{L_x^\frac{3}{s}L_v^2}\norm{f }_\nu
\\&\quad\le  \norm{f}_{L_v^2L_x^\frac{3}{s}}\left(\norm{\{{\bf I- P}\}f }_\nu+\norm{f}\right).
\end{split}
\end{equation}
We bound the first term in \eqref{Hs es 5} as, since $3/s>2$, by Sobolev's inequality,
\begin{equation}\label{Hs es 6}
\begin{split}
\norm{f}_{L_v^2L_x^\frac{3}{s}}\norm{\{{\bf I- P}\}f }_\nu&\lesssim \norm{f}_{L_v^2H_x^2} \norm{\{{\bf I-P}\}f }_\nu \lesssim \delta  \norm{\{{\bf
I-P}\}f }_\nu.
\end{split}
\end{equation}
While for the other term in \eqref{Hs es 5}, we shall separate the estimates according to the value of $s$. If $0< s\le 1/2$, then $3/s\ge 6$, we use
the Sobolev interpolation and Young's inequality to have
\begin{equation}\label{Hs es 7}
\norm{f}_{L_v^2L_x^\frac{3}{s}}\norm{f} \le  \norm{ \nabla f }^{1+s/2}\norm{ \nabla^2 f }^{1-s/2}\norm{ f } \lesssim \delta \left(  \norm{ \nabla f
}+ \norm{\nabla^2 f}\right);
\end{equation}
and if $s\in(1/2,3/2)$, then $2<3/s<6$, we use the (different) Sobolev interpolation and H\"older's inequality to have
\begin{equation}\label{Hs es 8}
\norm{f}_{L_v^2L_x^\frac{3}{s}}\norm{f}\lesssim \norm{ f } ^{s-1/2}\norm{\nabla f }^{3/2-s} \norm{ f }=\norm{ f } ^{s+1/2}\norm{\nabla f }^{3/2-s}.
\end{equation}

For the second term in \eqref{Hs es 0}, we do the splitting $f={\bf P_1}f+\{{\bf I-P_1}\}f$ to have, similarly as in \eqref{Hs es 5},
\begin{equation}\label{Hs es 2}
\begin{split}
&\left(\Lambda^{-s}\left(\nabla_x\Phi\cdot vg\right), \Lambda^{-s} f\right)
\\&\quad=\left(\Lambda^{-s}\left(\nabla_x\Phi\cdot vg\right), \Lambda^{-s}{\bf P_1}
f\right)+\left(\Lambda^{-s}\left(\nabla_x\Phi\cdot vg\right), \Lambda^{-s}\{{\bf I-P_1}\} f\right)
\\&\quad\lesssim
\norm{\Lambda^{-s}\left(|\nabla_x\Phi||g|_2\right)}\norm{\Lambda^{-s}f}+\norm{\Lambda^{-s}\left(|\nabla_x\Phi||g|_\nu\right)}\norm{\Lambda^{-s}\{{\bf
I-P_1}\} f}_\nu
\\&\quad\lesssim \norm{|\nabla_x\Phi||g|_2}_{L^p_x}\norm{\Lambda^{-s}f}+\norm{|\nabla_x\Phi||g|_\nu}_{L^p_x}\norm{\Lambda^{-s}\{{\bf
I-P_1}\} f}_\nu
\\&\quad\lesssim \norm{\nabla_x\Phi}_{L^{\frac{3}{s}}_x}\norm{g}\norm{\Lambda^{-s}f}+
\norm{\nabla_x\Phi}_{L^{\frac{3}{s}}_x}\norm{g}_\nu\norm{\Lambda^{-s}\{{\bf I-P_1}\} f}_\nu
\\&\quad\le  \norm{\nabla_x\Phi}_{H^2_x}\norm{g}\norm{\Lambda^{-s}f}+C\norm{\nabla_x\Phi}_{H^2_x}^2\norm{g}_\nu^2
+\frac{\sigma_0}{4}\norm{\Lambda^{-s}\{{\bf I-P_1}\} f}_\nu^2.
\end{split}
\end{equation}
For the last term in \eqref{Hs es 0}, we do not need the splitting and similarly we have
\begin{equation}\label{Hs es 3}
(\Lambda^{-s}\left(\nabla_x\Phi\cdot\nabla_vg\right), \Lambda^{-s} f)
 \lesssim\norm{\Lambda^{-s}\left(|\nabla_x\Phi||\nabla_vg|_2\right)}\norm{\Lambda^{-s}f}
\lesssim \norm{\nabla_x\Phi}_{H^2_x}\norm{\nabla_vg}\norm{\Lambda^{-s}f}.
\end{equation}

Consequently, in light of \eqref{Hs es 4}--\eqref{Hs es 3} and the definitions of $\mathcal{E}_N$ and $\mathcal{D}_N$, we deduce from \eqref{Hs es 0}
that \eqref{H-s1} holds for $s\in(0,1/2]$ and that \eqref{H-s2} holds for $s\in(1/2,3/2)$.
\end{proof}

\subsection{Proof of Theorems \ref{theorem1}}\label{sec 2.3}

In this subsection, we will combine all the energy estimates that we have derived in the previous two subsections and the interpolation between negative and positive Sobolev norms to prove Theorem \ref{theorem1}.
Note that for $\ell=0,1$, we can then absorb the  right hand side of \eqref{energy estimate 4} to obtain, by adjusting again the constant in the
definition of $\widetilde{\mathcal{E}}_\ell(t)$,
\begin{equation}\label{energy estimate 5}
\begin{split}
&\frac{d}{dt}\widetilde{\mathcal{E}}_\ell+\sum_{\ell\le k\le N}\norm{\left[\nabla^{k}\{{\bf I-P _1}\} f,\nabla^{k}\{{\bf I-P _2}\} g\right] }_\nu^2
\\&\quad+\sum_{\ell+1\le k\le N} \norm{\nabla^{k}  {\bf P_1 }f }^2+\sum_{\ell\le k\le N} \norm{\nabla^{k} {\bf
P_2 }g}^2 +\sum_{\ell\le k\le N+1} \norm{\nabla^{k} \nabla_x\Phi }^2\le 0\text{ for }\ell=0,1.
\end{split}
\end{equation}
In particular, taking $\ell=0$ in \eqref{energy estimate 5} and then multiplying by a large number $K>0$, adding the resulting estimate with
\eqref{I-P11}, we obtain
 \begin{equation}\label{111}
\begin{split}
\frac{d}{dt} \left(K\widetilde{\mathcal{E}}_0+\sum_{|\gamma|+|\beta|\le N\atop |\beta|\ge1}\norm{\left[\partial^\gamma_\beta\{{\bf I-P_1
}\}f,\partial^\gamma_\beta\{{\bf I-P_2 }\}g\right]}^2\right) + C_6\mathcal{D}_N\le 0. \end{split}
\end{equation}
We may define $C_6^{-1}$ times the expression under the time differentiation in \eqref{111} to be the instant energy functional $\mathcal{E}_N$, then
we have
\begin{equation}\label{energy di}
\frac{d}{dt}\mathcal{E}_N+ \mathcal{D}_N\le 0.
\end{equation}
Integrating \eqref{energy di} directly in time, we get \eqref{energy es}. Hence, if we assume $\mathcal{E}_N(0)\le \delta_0$ for a sufficiently small
$\delta_0>0$, then a standard continuity argument closes the a priori estimates that $\mathcal{E}_N(t)\le \delta$. Thus we can conclude the global
solution with the estimate \eqref{energy es} by the standard continuity argument of combining the local existence result and the a prior energy
estimates.

Now we turn to prove \eqref{H-sbound} and \eqref{decay0}. However, we are not able to prove them for all $s$ at this moment. We shall first prove
them for $s\in[0,1/2]$.

\begin{proof}[Proof of \eqref{H-sbound}--\eqref{decay0} for {$s\in[0,1/2]$}]
First, integrating in time the estimate \eqref{H-s1} of Lemma \ref{lemma H-s}, by the bound \eqref{energy es}, we obtain that for $s\in (0,1/2]$,
\begin{equation}
\begin{split}
\norm{\Lambda^{-s}f(t)}^2 &\le \norm{\Lambda^{-s}f_0}^2 +C\int_0^t \left(\norm{\Lambda^{-s}f(\tau)}+1\right)\mathcal{D}_N(\tau)\,d\tau
\\&\le C_0\left(1+\sup_{0\le \tau \le t}\norm{\Lambda^{-s}f(\tau)}\right).
\end{split}
\end{equation}
This together with \eqref{energy es} gives \eqref{H-sbound} for $s\in[0,1/2]$.

Next, we take $\ell=0,1$ in \eqref{energy estimate 5} and recall from the definition of the energy functional $\widetilde{\mathcal{E}}_\ell(t)$ that
there is only one exceptional term $\norm{\nabla^\ell {\bf P_1}f(t)}^2\le \norm{\nabla^\ell  f}^2$ that can not be bounded by the corresponding
dissipation in \eqref{energy estimate 5}. The key point is to do the interpolation between the negative and positive Sobolev norms by using Lemma
\ref{-sinte},
\begin{equation}\label{inter 1}
\norm{\nabla^\ell  f}\le C\norm{\nabla^{\ell+1} f}^{\frac{\ell+s}{\ell+1+s}}\norm{\Lambda^{-s} f}^{\frac{1}{\ell+1+s}}.
\end{equation}
This together with the bound \eqref{H-sbound} yields that there exists $C_0>0$ such that
\begin{equation}\label{inter 3}
\norm{\nabla^{\ell+1}  f}\ge C\norm{\nabla f}^{1+\frac{1}{\ell+s}}\norm{\Lambda^{-s} f}^{-\frac{1}{\ell+s}}\ge C_0\norm{\nabla^\ell  f}^{
1+\frac{1}{\ell+s}}.
\end{equation}
Hence, by \eqref{inter 3} and \eqref{energy es}, we deduce from \eqref{energy estimate 5} that
\begin{equation}\label{time inequality}
\frac{d}{dt}\widetilde{\mathcal{E}}_\ell+ C_0\left(\widetilde{\mathcal{E}}_\ell\right)^{1+\frac{1}{\ell+s}}\le 0\ \text{ for }\ \ell+s\neq0.
\end{equation}
Solving this inequality directly and \eqref{energy es} again, we obtain
\begin{equation}\label{ell}
\widetilde{\mathcal{E}}_\ell(t) \le \left({\widetilde{\mathcal{E}}_\ell(0)}^{-1/(\ell+s)}+C_0(\ell+s) t\right)^{-(\ell+s)}\le C_0 (1+ t)^{-(\ell+s)}\
\text{ for }\ \ell+s\neq0.
\end{equation}
Taking $\ell=1$ in \eqref{ell} together with \eqref{energy es}, we obtain \eqref{decay0}. Also note that there is only one exceptional term
$\norm{{\bf P_1}f(t)}^2$ in $\mathcal{E}_N$ that can not be bounded by $\mathcal{D}_N$, so using the same arguments leading to \eqref{time
inequality} with $\ell=0$, we can obtain
\begin{equation}
\mathcal{E}_N(t) \le \left(\mathcal{E}_N(0)^{-1/s}+C_0s t\right)^{-s}\le C_0 (1+ t)^{-s}\ \text{ for }\ s>0.
\end{equation}
This together with \eqref{energy es} again gives \eqref{decay00}.
\end{proof}

Now we can present the
\begin{proof}[Proof of \eqref{H-sbound}--\eqref{decay0} for {$s\in(1/2,3/2)$}]
Notice that the arguments for the case $s\in[0,1/2]$ can not be applied to this case. However, observing that we have $f_0\in L^2_v\dot{H}_x^{-1/2}$
since $L^2_v\dot{H}_x^{-s}\cap L^2_vL^2_x\subset L^2_v\dot{H}_x^{-s'}$ for any $s'\in [0,s]$, we then deduce from what we have proved for
\eqref{H-sbound}--\eqref{decay0} with $s=1/2$ that the following decay result holds:
\begin{equation}\label{decay11}
\sum_{\ell\le k\le N}\norm{\nabla^k f(t)}^2\le C_0(1+t)^{-(\ell+ \frac{1}{2})}\, \hbox{ for } \ell=0, 1.
\end{equation}
Hence, by \eqref{decay11} and \eqref{energy es}, we deduce from \eqref{H-s2} that for $s\in (1/2,3/2)$,
\begin{equation}\label{sssbound}
\begin{split}
\norm{\Lambda^{-s}f(t)}^2&\le \norm{\Lambda^{-s}f_0}^2+C \int_0^t\left(\left(\norm{\Lambda^{-s}f(\tau)}+1\right)\mathcal{D}_N(\tau)+\norm{ f(\tau)
}^{2s+1}\norm{ \nabla f (\tau)  }^{3-2s}\right)\,d\tau.
\\&\le C_0+C_0\sup_{0\le \tau \le t}\norm{\Lambda^{-s}f(\tau)}+C_0\int_0^t(1+\tau)^{-(5/2-s)}\,d\tau
\\&\le C_0\left(1+\sup_{0\le \tau \le t}\norm{\Lambda^{-s}f(\tau)}\right) .
\end{split}
\end{equation}
This proves \eqref{H-sbound} for $s\in (1/2,3/2)$, and we may then repeat the arguments leading to \eqref{decay00}--\eqref{decay0} for $s\in [0,1/2]$
to obtain \eqref{decay00}--\eqref{decay0} for $s\in (1/2,3/2)$. The proof of Theorem \ref{theorem1} is completed.
\end{proof}

\section{Weighted energy estimates and proof of Theorem \ref{theorem2}}\label{sec 3}

\subsection{Weighted energy estimates}\label{3.1}
In this subsection, we will derive the weighted energy estimates on the spatial derivatives
of the solution, and then with the help of these weighted norms we will
derive some further energy estimates compared to the basic energy estimates derived in Section \ref{sec 2.1}. The following lemma provides the weighted energy evolution of $\left[\{{\bf I-P_1}\}f,\{{\bf I-P_2}\}g\right]$.
\begin{lemma}\label{lemma micro 11}
If $\mathcal{E}_N(t)\le \delta$, then for $k=0,\dots,N-1$, we have
\begin{equation}\label{micro estimate 000}
\frac{d}{dt}\norm{\left[\{{\bf I-P_1}\}f,\{{\bf I-P_2}\}g\right]}_\nu^2 +C\norm{\left[\nu\{{\bf I-P_1}\}f,\nu\{{\bf I-P_2}\}g\right] }^2\lesssim
\mathcal{D}_N
\end{equation}
and
\begin{equation}\label{micro estimate 11}
\begin{split}
&\frac{d}{dt}\sum_{1\le k\le N-1}\norm{\left[\nabla^k\{{\bf I-P_1}\}f,\nabla^k\{{\bf I-P_2}\}g\right]}_\nu^2
\\&\quad+C\sum_{1\le k\le
N-1}\norm{\left[\nu\nabla^k\{{\bf I-P_1}\}f,\nu\nabla^k\{{\bf I-P_2}\}g\right] }^2\lesssim  \mathcal{D}_N.
\end{split}
\end{equation}
\end{lemma}
\begin{proof}
We only illustrate the estimate on $\{{\bf I-P_2 }\} g$. Applying $\partial^\gamma$ (with $0\le |\gamma|=k\le N-1$) to $\eqref{VPB_per}_2$ and then
taking the inner product with $\nu\partial^\gamma\{{\bf I-P_2 }\} g $, we obtain
\begin{equation}\label{I-P111}
\begin{split}
&\frac{1}{2}\frac{d}{dt}\norm{\partial^\gamma\{{\bf I-P_2 }\}g}_\nu^2 + \left(\partial^\gamma\mathcal{L}_2 \{{\bf  I-P_2
}\}g,\nu\partial^\gamma\{{\bf I-P_2 }\}g\right)
\\&\quad+\left(\partial_t\partial^\gamma{\bf P_2 }g+v\cdot\nabla_x \partial^\gamma{\bf P_2}g
-\partial^\gamma\nabla_x\Phi\cdot  v\sqrt{\mu},\nu\partial^\gamma \{{\bf I-P_2 }\}g\right)
\\&\quad =  \left(\partial^\gamma \left(\Gamma(g,f)+\frac{1}{2}\nabla_x\Phi\cdot vf-\nabla_x\Phi\cdot\nabla_vf\right),\nu\partial^\gamma \{{\bf I-P_2 }\}g\right).
 \end{split}
\end{equation}

By the linear estimate \eqref{4} in Lemma \ref{linear c}, we know
\begin{equation}
\left(\mathcal{L}_2 \partial^\gamma\{{\bf I-P_2}\}g ,\nu\partial^\gamma \{{\bf I-P_2}\}g \right)\ge \frac{1}{2}\norm{\nu\partial^\gamma \{{\bf
I-P_2}\}g}^2-C\norm{\partial^\gamma \{{\bf I-P_2}\}g}_\nu^2.
\end{equation}
The second line in \eqref{I-P111} can be bounded by, as in Lemma \ref{lemma spatial velocity},
\begin{equation}
 \frac{1}{4}\norm{\partial^\gamma \{{\bf I-P_2}\}g }_\nu^2+C\norm{\nabla_x\partial^\gamma g }^2+C\norm{\partial^\gamma \nabla_x\Phi }^2.
\end{equation}

Now we turn to the third line in \eqref{I-P111}. First, by the estimate \eqref{nonlineares3} (with $\eta=0$) of Lemma \ref{nonlinearcol1} and applying the similar argument that for $J_1$, we have
\begin{equation}
\begin{split}
&\left(\partial^\gamma\Gamma(g,f),\nu\partial^\gamma\{{\bf I-P_2 }\}g\right)=\sum_{\gamma_1\le \gamma}C_\gamma^{\gamma_1}(\Gamma(\partial^{\gamma_1}
g,
\partial^{\gamma-\gamma_1} f),\nu\partial^\gamma \{{\bf I-P_2}\}g)
\\&\quad\lesssim\sum_{\gamma_1\le \gamma}\norm{\Gamma(\partial^{\gamma_1} g,
\partial^{\gamma-\gamma_1} f)}\norm{\nu\partial^\gamma \{{\bf I-P_2}\}g}
\\&\quad\lesssim\sum_{\gamma_1\le \gamma}\norm{|\nabla^{|\gamma_1|} g|_2|\nu\nabla^{k-|\gamma_1|} f|_2
+|\nu\nabla^{|\gamma_1|} g|_2|\nabla^{k-|\gamma_1|} f|_2}\norm{\nu\nabla^k\{{\bf I- P_2}\}g}
\\&\quad\lesssim
\delta\left(\norm{\left[\nabla^{k+1}f,\nabla^{k+1}g\right]}+\norm{\nu\nabla^k\{{\bf I- P_2}\}g}+\sum_{1\le \ell\le N-1}\norm{\left[\nu\nabla^\ell
\{{\bf I-P_1}\}f,\nu\nabla^\ell \{{\bf I-P_2}\}g\right]}^2\right).
\end{split}
\end{equation}
Comparing this with \eqref{g non 11}, the only difference is that we replace the $\nu^{1/2}$-weighted norm by the $\nu$-weighted norm. This
observation is also valid for the estimates of the last two term in the third line in \eqref{I-P111}. Hence, we may easily complete the estimates for
$\{{\bf I-P_2 }\} g$ with this repalcement. Applying the similar argument and the observation to $\{{\bf I-P_1 }\} f$, and summing over $|\gamma|=k$,
we get
\begin{equation}\label{micro estimate 1111}
\begin{split}
&\frac{1}{2}\frac{d}{dt}\norm{\left[\nabla^k\{{\bf I-P_1}\}f,\nabla^k\{{\bf I-P_2}\}g\right]}_\nu^2+\frac{1}{4}\norm{\left[\nu\nabla^k\{{\bf
I-P_1}\}f,\nu\nabla^k\{{\bf I-P_2}\}g\right] }^2
\\&\quad\lesssim   \norm{\left[\nabla^{k+1} f,\nabla^{k+1} g\right]}^2+\norm{\nabla^{k}g}_\nu^2
+\norm{\left[\nabla^k\nabla_x\Phi,\nabla^{k+1}\nabla_x\Phi\right]}^2+\sum_{2\le \ell \le N-1}\norm{\nabla^\ell\nabla_x\Phi}^2
\\&\qquad+\delta\sum_{1\le \ell \le N-1}\norm{\left[\nu\nabla^{\ell}\{{\bf I- P_1}\}f,\nu\nabla^{\ell}\{{\bf I-
P_2}\}g\right]}^2.
\end{split}
\end{equation}
Summing the above up $k$ from $1$ to $N-1$, by the definition of $\mathcal{D}_N$, since $\delta$ is small, we obtain \eqref{micro estimate 11}. The estimate \eqref{micro estimate 000}
follows similarly and we omit the details.
\end{proof}

By \eqref{micro estimate 000}, we know that if $\mathcal{E}_N(0)+\norm{[f_0,g_0]}_\nu^2$ and is small, then $\mathcal{E}_N(t)+\norm{[f(t),g(t)]}_\nu^2$ is small. With the help of this
weighted bound, we can improve the energy estimates derived in Section \ref{sec 2.1}.
\begin{lemma}\label{lemma micro 22}
If $\mathcal{E}_N(t)+\norm{\left[f(t),g(t)\right]}_\nu^2\le \delta$, then for $k=0,\dots, N-1$, we have
\begin{equation}\label{improved energy estimate 1}
\begin{split}
&\frac{d}{dt}\left(\norm{\left[\nabla^k f,\nabla^k g\right]}^2+\norm{\nabla^k \nabla_x\Phi}^2\right) +C\norm{\left[\nabla^k\{{\bf I-P_1}\}
f,\nabla^k\{{\bf I-P_2}\} g\right]}_\nu^2
\\&\quad \lesssim \delta \left(\norm{\left[\nabla^{k+1} f,\nabla^{k+1} g\right]}^2+\norm{\nabla^{k}{\bf P_2}g}^2
+\norm{\left[\nabla^k\nabla_x\Phi,\nabla^{k+1}\nabla_x\Phi\right]}^2+\sum_{2\le \ell \le N}\norm{\nabla^\ell\nabla_x\Phi}^2\right);
\end{split}
\end{equation}
and for $k=N$, we have
\begin{equation}\label{improved energy estimate 2}
\begin{split}
&\frac{d}{dt}\left(\norm{\left[\nabla^N f,\nabla^N g\right]}^2+\norm{\nabla^N \nabla_x\Phi}^2\right) +C\norm{\left[\nabla^N\{{\bf I-P_1}\}
f,\nabla^N\{{\bf I-P_2}\} g\right]}_\nu^2
\\&\quad\lesssim\delta \left(\norm{\left[\nabla^N f,\nabla^N g\right]}^2+\sum_{2\le \ell \le N}\norm{\nabla^\ell\nabla_x\Phi}^2\right).
\end{split}
\end{equation}
\end{lemma}
\begin{proof}
Comparing \eqref{improved energy estimate 1}--\eqref{improved energy estimate 2} with \eqref{spatial energy estimate 1}--\eqref{spatial energy
estimate 2} of Lemma \ref{lemma spatial energy}, the only difference is that we remove the last $\nu^{1/2}$-weighted summing term from \eqref{spatial
energy estimate 1}--\eqref{spatial energy estimate 2}. Hence clearly, we only need to improve the estimates of those terms in the proof of Lemma
\ref{lemma spatial energy} that leads to this $\nu^{1/2}$-weighted summing term. Indeed, these terms are $J_{1},J_2$ and $I_1,I_2$. We shall only
illustrate the improved estimates of $J_1$, and the other terms can be treated in the same way.  More precisely, we shall revisit the term $J_{12}$,
\begin{equation}
J_{12}=\norm{|\nabla^{|\gamma_1|} g|_2|\nabla^{k-|\gamma_1|}\{{\bf I-P_1}\} f|_\nu}.
\end{equation}

Note that now we can also bound the $\nu$-weighted factor by the energy, so we can pursue to adjust the index. For $k=0,\dots,N-1$, if
$|\gamma_1|=0$, then we have
\begin{equation}
J_{12}\lesssim \norm{g}_{L^\infty_xL^2_v}\norm{\nabla^k\{{\bf I-P_1}\}f}_{\nu}  \lesssim \delta \norm{\nabla^{k}\{{\bf I-P_1}\}f}_\nu;
\end{equation}
if $|\gamma_1|\ge 1$, then by Lemma \ref{Minkowski} and Lemma \ref{interpolation}, we have
\begin{equation}
\begin{split}
J_{12}&\lesssim \norm{\nabla^{|\gamma_1|} g}_{L^3_xL^2_v}\norm{\nu^{1/2} \nabla^{k-|\gamma_1|}\{{\bf I-P_1}\}f}_{L^6_xL^2_v}
\\&\lesssim \norm{\nabla^{|\gamma_1|} g}_{L^2_vL^3_x}\norm{\nu^{1/2}
\nabla^{k-|\gamma_1|}\{{\bf I-P_1}\}f}_{L^2_vL^6_x}
\\&\lesssim \norm{\nabla^\alpha g}^{1-\frac{|\gamma_1|-1}{k}}\norm{\nabla^{k+1}g}^{\frac{|\gamma_1|-1}{k}}
\norm{\{{\bf I-P_1}\}f}_\nu^{\frac{|\gamma_1|-1}{k}}\norm{\nabla^{k}\{{\bf I-P_1}\}f}_\nu^{1-\frac{|\gamma_1|-1}{k}}
\\&\lesssim \delta\left(\norm{\nabla^{k+1}g}+\norm{\nabla^{k}\{{\bf I-P_1}\}f}_\nu\right),
\end{split}
\end{equation}
where we have denoted $\alpha$ by
\begin{equation}
\begin{split}
&\frac{1}{3}-\frac{|\gamma_1|}{3}= \left(\frac{1}{2}-\frac{\alpha}{3}\right)\times
\left(1-\frac{|\gamma_1|-1}{k}\right)+\left(\frac{1}{2}-\frac{k+1}{3}\right)\times \frac{|\gamma_1|-1}{k}
\\&\quad\Longrightarrow \alpha=\frac{\frac{3}{2}k-(|\gamma_1|-1)}{k-(|\gamma_1|-1)}\le \frac{k}{2}+1.
\end{split}
\end{equation}
Hence, we have that for $k=0,\dots,N-1$,
\begin{equation}\label{j1211}
 J_{12}\lesssim \delta\left(\norm{\nabla^{k+1}g}+\norm{\nabla^{k}\{{\bf I-P_1}\}f}_\nu\right).
\end{equation}

Now for $k=N$, if $|\gamma_1|\ge N-1$, by Lemma \ref{Minkowski} and Lemma \ref{interpolation}, we estimate
\begin{equation}
\begin{split}
J_{12}&\le \norm{\nabla^{|\gamma_1|} g}\norm{\nu^{1/2}\nabla^{N-|\gamma_1|}\{{\bf I-P_1}\} f}_{L^\infty_xL^2_v} \le\norm{\nabla^{|\gamma_1|}
g}\norm{\nu^{1/2}\nabla^{N-|\gamma_1|}\{{\bf I-P_1}\} f}_{L^2_vL^\infty_x}
\\&\lesssim \norm{\nabla^\alpha g}^{1-\frac{2|\gamma_1|-3}{2N}}\norm{\nabla^Ng}^{\frac{2|\gamma_1|-3}{2N}}\norm{\{{\bf I-P_1}\}f}_{\nu}^{\frac{2|\gamma_1|-3}{2N}}\norm{\nabla^N\{{\bf I-P_1}\}f}_{\nu}^{1-\frac{2|\gamma_1|-3}{2N}}
\\&\lesssim \delta\left(\norm{\nabla^Ng}+\norm{\nabla^N\{{\bf I-P_1}\}f}_{\nu}\right),
\end{split}
\end{equation}
where we have denoted $\alpha$ by
\begin{equation}
\begin{split}
&|\gamma_1|=\alpha\times \left(1-\frac{2|\gamma_1|-3}{2N}\right)+N\times\frac{2|\gamma_1|-3}{2N}
\\&\quad\Longrightarrow \alpha=\frac{3N}{2(N-|\gamma_1|)+3}\le N;
\end{split}
\end{equation}
and if $|\gamma_1|\le N-2$, by again Lemma \ref{Minkowski} and Lemma \ref{interpolation}, we estimate
\begin{equation}
\begin{split}
J_{12}&\le \norm{\nabla^{|\gamma_1|} g}_{L^\infty_xL^2_v}\norm{\nabla^{N-|\gamma_1|}\{{\bf I-P_1}\} f}_{\nu} \le\norm{\nabla^{|\gamma_1|}
g}_{L^2_vL^\infty_x}\norm{\nabla^{N-|\gamma_1|}\{{\bf I-P_1}\} f}_{\nu}
\\&\lesssim \norm{\nabla^\alpha g}^{1-\frac{|\gamma_1|}{N}}\norm{\nabla^{N} g}^{\frac{|\gamma_1|}{N}}
\norm{\{{\bf I-P_1}\} f}_\nu^{\frac{|\gamma_1|}{N}}\norm{\nabla^{N}\{{\bf I-P_1}\} f}_\nu^{\frac{N-|\gamma_1|}{N}}
\\&\lesssim \delta\left(\norm{\nabla^Ng}+\norm{\nabla^N\{{\bf I-P_1}\}f}_{\nu}\right),
\end{split}
\end{equation}
where we have denoted $\alpha$ by
\begin{equation}
\begin{split}
&-\frac{|\gamma_1|}{3}=
\left(\frac{1}{2}-\frac{N}{3}\right)\times\frac{|\gamma_1|}{N}+\left(\frac{1}{2}-\frac{\alpha}{3}\right)\times\left(1-\frac{|\gamma_1|}{N}\right)
\\&\quad\Longrightarrow \alpha=\frac{3N}{2(N-|\gamma_1|)}\le \frac{3N}{4}\ \text{ since }|\gamma_1|\le N-2.
\end{split}
\end{equation}
Hence, we have that for $k=N$,
\begin{equation}\label{j1212}
 J_{12}\lesssim \delta\left(\norm{\nabla^N g}+\norm{\nabla^N\{{\bf I-P_1}\}f}_\nu\right).
\end{equation}

Comparing the estimates \eqref{j1211} and \eqref{j1212} with \eqref{j12}, we have succeeded in removing the $\nu^{1/2}$-weighted summing term from
the estimates of $J_{12}$. Applying the similar argument, we can remove this $\nu^{1/2}$-weighted summing term from the estimates of $J_2$ and
$I_1,I_2$, and hence we get \eqref{improved energy estimate 1}--\eqref{improved energy estimate 2}.
\end{proof}

By the estimates \eqref{micro estimate 000}--\eqref{micro estimate 11}, we know that if $\mathcal{E}_N(0)+\sum_{0\le k\le
N-1}\norm{\left[\nabla^kf_0,\nabla^kg_0\right]}_\nu^2$ is small, then $\mathcal{E}_N(t)+\sum_{0\le k\le N-1}\norm{\left[\nabla^kf(t),\nabla^kg(t)\right]}_\nu^2$ is
small. With the help of this weighted bound,  we can deduce a further energy estimate of $g$ and $\nabla_x\Phi$ which implies the exponential decay
of $g$ and $\nabla_x\Phi$. This energy estimate also can be used to kill the summing term related to $\nabla_x\Phi$ in the right hand side of
\eqref{improved energy estimate 1}--\eqref{improved energy estimate 2}, and hence the energy estimates in Lemma \ref{lemma micro 22} will be improved.
\begin{lemma}\label{lemma micro 33}
If $\mathcal{E}_N(t)+\sum_{0\le k\le N-1}\norm{\left[\nabla^kf(t),\nabla^kg(t)\right]}_\nu^2\le \delta$, then there exists an equivalent energy
functional
\begin{equation}
\mathcal{E}_g\sim\sum_{0\le k\le N-1}\|\nabla^k  g \|^2+\sum_{0\le k\le N}\norm{\nabla^k \nabla_x\Phi}^2
\end{equation}
 such that the following inequality holds:
\begin{equation}\label{micro g}
\frac{d}{dt}\mathcal{E}_g+ C\left(\sum_{0\le k\le N-1}\|\nabla^k   g \|_\nu^2+\sum_{0\le k\le N}\norm{\nabla^k \nabla_x\Phi}^2\right)\le0.
\end{equation}
\end{lemma}
\begin{proof}
The standard $\nabla^k$ energy estimates on $\eqref{VPB_per}_2$ yields that for $k=0,\dots,N-1,$
\begin{equation}\label{gggg}
\begin{split}
&\frac{1}{2}\frac{d}{dt}\left(\|\nabla^k g \|^2+\norm{\nabla^k \nabla_x\Phi}^2\right) + \sigma_0\norm{\nabla^k\{{\bf I-P_2}\}g}_\nu^2 \\&\quad\le
\left(\nabla^k\left(\Gamma(g ,f )+\frac{1}{2}\nabla_x\Phi\cdot vf-\nabla_x\Phi\cdot\nabla_vf\right),\nabla^k g\right ).
\end{split}
\end{equation}
To estimate the right hand side of \eqref{gggg}, we will simply bound the norm of the $f$-related factors by $\delta$. More precisely, by the
collision estimate \eqref{nonlineares3} (with $\eta=0$) and Sobolev's inequality, we have
\begin{equation}
\begin{split}
\left(\nabla^k\Gamma(g ,f ),\nabla^k g\right )&\lesssim \sum_{0\le \ell\le k}\norm{|\nabla^\ell g|_\nu|\nabla^{k-\ell} f|_2+|\nabla^\ell
g|_2|\nabla^{k-\ell} f|_\nu}\norm{\nabla^k g}_\nu \\&\lesssim \sum_{0\le k\le N-1}\norm{\nabla^k f}_\nu\sum_{0\le k\le N-1}\norm{\nabla^k
g}_\nu\norm{\nabla^k g}_\nu\\&\lesssim \delta \sum_{0\le k\le N-1}\norm{\nabla^k g}_\nu^2.
\end{split}
\end{equation}
Applying the similar argument to the other two terms, we get
\begin{equation}\label{000}
\begin{split}
&\frac{d}{dt}\sum_{0\le k\le N-1}\left(\|\nabla^k g \|^2+\norm{\nabla^k \nabla_x\Phi}^2\right) + C\sum_{0\le k\le N-1}\norm{\nabla^k\{{\bf
I-P_2}\}g}_\nu^2
\\&\quad\lesssim \delta\sum_{0\le k\le
N-1}\left(\norm{\nabla^k \nabla_x\Phi}^2+\|\nabla^k {\bf P_2}g \|^2\right).
\end{split}
\end{equation}

On the other hand, we may go back to \eqref{full_estimate 2} to find that for $k=0,\dots,N-2$
\begin{equation}\label{gkgk}
\begin{split}
&\frac{d}{dt}G_g^k(t)+\|\nabla^{k}  {\bf P }_2g \|^2+\|\nabla^{k+1}{\bf P }_2g \|^2+\norm{[\nabla^k \nabla_x \Phi,\nabla^{k+1}\nabla_x
\Phi,\nabla^{k+2}\nabla_x \Phi]}^2
\\&\quad\lesssim\|\nabla^{k}\{{\bf I-P_2}\} g \|^2+\|\nabla^{k+1}\{{\bf I-P_2}\} g]\|^2
+\norm{\nabla^{k}\mathfrak{N}_{2,\parallel}}^2.
\end{split}
\end{equation}
To estimate $\norm{\nabla^{k}\mathfrak{N}_{2,\parallel}}^2$, we simply bound the norm of the $f$-related terms by $\delta$ to have
\begin{equation}
\norm{ \nabla^k\mathfrak{N}_{2,\parallel}}^2 \lesssim \delta \sum_{0\le k\le N-2}\left(\norm{\nabla^k \nabla_x\Phi}^2+\|\nabla^k g \|^2\right).
\end{equation}
Then summing up \eqref{gkgk} from $k=0$ to $N-2$, since $\delta$ is small, we have
\begin{equation}\label{333}
\frac{d}{dt}\sum_{0\le k\le N-2}G_g^k(t)+\sum_{0\le k\le N-1}\|\nabla^{k}  {\bf P }_2g \|^2+\sum_{0\le k\le N}\norm{\nabla^k \nabla_x \Phi}^2
\lesssim \sum_{0\le k\le N-1}\|\nabla^{k}\{{\bf I-P_2}\} g \|^2.
\end{equation}
A suitable linear combination of \eqref{000} and \eqref{333} gives \eqref{micro g}.
\end{proof}

The following lemma  provides the needed estimates for proving the faster decay of the microscopic part $\{{\bf I-P_1}\}f$.
\begin{lemma}\label{lemma micro 44}
If $\mathcal{E}_N(t)+\sum_{0\le k\le N-1}\norm{\left[\nabla^kf(t),\nabla^kg(t)\right]}_\nu^2\le \delta$, then for $k=0,\dots,N-2$,
\begin{equation}\label{micro estimate 2}
\begin{split}
&\frac{d}{dt}\norm{\left[\nabla^k\{{\bf I-P_1}\}f,\nabla^k\{{\bf I-P_2}\}g\right]}^2+C\norm{\left[\nabla^k\{{\bf I-P_1}\}f,\nabla^k\{{\bf
I-P_2}\}g\right]}_\nu^2
\\&\quad\lesssim\norm{\left[\nabla^{k+1} f,\nabla^{k+1} g\right]}^2+\norm{\nabla^{k}{\bf P_2}g}^2+\sum_{0\le k\le N-2}\norm{\nabla^k \nabla_x\Phi}^2.
\end{split}
\end{equation}
\end{lemma}
\begin{proof}
The proof of \eqref{micro estimate 2} is similar to those of Lemma \ref{lemma micro 11}--\ref{lemma micro 33} and hence is omitted.
\end{proof}

\subsection{Proof of Theorem \ref{theorem2}}\label{3.2}
In this subsection, we will complete the proof of Theorem \ref{theorem2} by using instead the weighted estimates derived in the previous subsection. Recall
that all the statements of Theorem \ref{theorem1} are valid.

First, letting $\mathcal{E}_N(0)+\norm{[f_0,g_0]}_\nu^2$ be sufficiently small, then by the estimate \eqref{micro estimate 000} of Lemma \ref{lemma micro 11}, we have that
$\mathcal{E}_N(t)+\norm{\left[f(t),g(t)\right]}_\nu^2$ is small. Hence, in light of the estimates in Lemma \ref{lemma micro 22}, we can improve the
estimates \eqref{energy estimate 4} to be that for $\ell=0,\dots,N-1$,
\begin{equation}\label{energy estimate 4'}
\begin{split}
&\frac{d}{dt}\widetilde{\mathcal{E}}_\ell+\sum_{\ell\le k\le N}\|[\nabla^{k}\{{\bf I-P _1}\} f,\nabla^{k}\{{\bf I-P _2}\} g] \|_\nu^2
+\sum_{\ell+1\le k\le N} \|\nabla^{k}  {\bf P_1 }f \|^2\\&\quad+\sum_{\ell\le k\le N} \norm{\nabla^{k} {\bf P_2 }g}^2 +\sum_{\ell\le k\le N+1}
\|\nabla^{k} \nabla_x\Phi \|^2\le C_5\delta \sum_{2\le k \le N+1}\norm{\nabla^k\nabla_x\Phi}^2.
\end{split}
\end{equation}
Taking $\ell=2$ in \eqref{energy estimate 4'}, we can then absorb the  right hand side to obtain, by adjusting again the constant in the definition
of $\widetilde{\mathcal{E}}_2(t)$,
\begin{equation}\label{energy estimate 5'}
\begin{split}
&\frac{d}{dt}\widetilde{\mathcal{E}}_2+\sum_{2\le k\le N}\norm{\left[\nabla^{k}\{{\bf I-P _1}\} f,\nabla^{k}\{{\bf I-P _2}\} g\right] }_\nu^2
\\&\quad+\sum_{3\le k\le N} \norm{\nabla^{k}  {\bf P_1 }f }^2+\sum_{2\le k\le N} \norm{\nabla^{k} {\bf
P_2 }g}^2 +\sum_{2\le k\le N+1} \norm{\nabla^{k} \nabla_x\Phi }^2\le 0.
\end{split}
\end{equation}
Then the time differential inequality \eqref{time inequality} corresponds to
\begin{equation}\label{time inequality'}
\frac{d}{dt}\widetilde{\mathcal{E}}_2+ C_0\left(\widetilde{\mathcal{E}}_2\right)^{1+\frac{1}{2+s}}\le 0.
\end{equation}
Solving this inequality and recalling the definition of $\widetilde{\mathcal{E}}_2$, we deduce \eqref{decay21}.

Now, letting $\mathcal{E}_N(0)+\sum_{0\le k\le N-1}\norm{[\nabla^kf_0,\nabla^kg_0]}_\nu^2$ be sufficiently small, then by the estimate \eqref{micro estimate 11} of Lemma \ref{lemma micro 11},
we have that $\mathcal{E}_N(t)+\sum_{0\le k\le N-1}\norm{\left[\nabla^kf(t),\nabla^kg(t)\right]}_\nu^2$ is small. Then by Lemma \ref{lemma micro
33}, there exists a constant $\lambda>0$ such that
\begin{equation}
\frac{d}{dt}\mathcal{E}_g+ \lambda \mathcal{E}_g\le0.
\end{equation}
Solving this inequality and recalling the definition of $\mathcal{E}_g$, we obtain \eqref{decay22}. On the other hand, adding the estimates
\eqref{energy estimate 4'} and \eqref{micro g}, since $\delta$ is small, we obtain
\begin{equation}\label{energy estimate 411'}
\begin{split}
&\frac{d}{dt}\left(\widetilde{\mathcal{E}}_\ell+\mathcal{E}_g\right)+\sum_{\ell\le k\le N}\|\nabla^{k}\{{\bf I-P _1}\} f\|_\nu^2 +\sum_{\ell+1\le
k\le N} \|\nabla^{k}  {\bf P_1 }f \|^2\\&\quad+\sum_{0\le k\le N} \norm{\nabla^{k} g}_\nu^2 +\sum_{0\le k\le N+1} \|\nabla^{k} \nabla_x\Phi \|^2\le
0.
\end{split}
\end{equation}
Notice that there is only one exceptional term ${\bf \nabla^\ell P_1 }f$ in $\widetilde{\mathcal{E}}_\ell+\mathcal{E}_g$ that can not be bounded by
the corresponding dissipation in \eqref{energy estimate 411'}. But similarly to \eqref{time inequality}, we have that for $\ell=3,\dots,N-1$,
\begin{equation}\label{time inequality'''}
\frac{d}{dt}\left(\widetilde{\mathcal{E}}_\ell+\mathcal{E}_g\right)+
C_0\left(\widetilde{\mathcal{E}}_\ell+\mathcal{E}_g\right)^{1+\frac{1}{\ell+s}}\le 0.
\end{equation}
Solving this inequality will in particular give \eqref{decay231} and \eqref{decay23}.

Finally, applying the Gronwall inequality to \eqref{micro estimate 2}, by \eqref{decay22} and \eqref{decay23}, we obtain that for $k=0,\dots,N-2$,
\begin{equation}
\begin{split}
&\norm{\left[\nabla^k\{{\bf I-P_1}\}f(t),\nabla^k\{{\bf I-P_2}\}g(t)\right]}^2
\\&\quad\le e^{-Ct}\norm{\left[\nabla^k\{{\bf I-P_1}\}f_0,\nabla^k\{{\bf
I-P_2}\}g_0\right]}^2 +C\int_0^te^{-C(t-\tau)}\left(\norm{\left[\nabla^{k+1} f(\tau),\nabla^{k+1} g(\tau)\right]}^2\right.
\\&\qquad\qquad\qquad\qquad\qquad\qquad\qquad\qquad\qquad\qquad\left.+\norm{\nabla^{k}{\bf
P_2}g(\tau)}^2+\sum_{0\le k\le N-2}\norm{\nabla^k \nabla_x\Phi(\tau)}^2\right)\,d\tau
\\&\quad\le C_0e^{-Ct}+C_0\int_0^te^{-C(t-\tau)} \left((1+\tau)^{-(k+1+s)}+e^{-\lambda \tau}\right)\,d\tau
\\&\quad\le C_0(1+t)^{-(k+1+s)}.
\end{split}
\end{equation}
This in particular gives \eqref{decay24}. The proof of Theorem \ref{theorem2} is completed.\hfill$\Box$

\appendix

\section{Analytic tools}\label{section_appendix}

\subsection{Sobolev type inequalities}

We will extensively use the Sobolev interpolation of the Gagliardo-Nirenberg inequality.
 \begin{lemma}\label{1interpolation}
 Let $0\le m, \alpha\le \ell$, then we have
\begin{equation}
\norm{\nabla^\alpha f}_{L^p}\lesssim \norm{  \nabla^mf}_{L^2}^{1-\theta}\norm{ \nabla^\ell f}_{L^2}^{\theta}
\end{equation}
where $0\le \theta\le 1$ and $\alpha$ satisfy
\begin{equation}
\frac{1}{p}-\frac{\alpha}{3}=\left(\frac{1}{2}-\frac{m}{3}\right)(1-\theta)+\left(\frac{1}{2}-\frac{\ell}{3}\right)\theta.
\end{equation}
\begin{proof}
This can be found in  \cite[pp. 125, THEOREM]{N1959}.
\end{proof}
\end{lemma}

We shall also use the corresponding Sobolev interpolation of the Gagliardo-Nirenberg inequality for the functions on $\r3_x\times\r3_v$.
\begin{lemma}
\label{interpolation}
Let $0\le m, \alpha\le \ell$. Let $w(v)$ be any weight function of $v$, then we have
\begin{equation}\label{GN00}
\left(\int_{\mathbb{R}^3_v}w\norm{\nabla^\alpha f}_{L^p_x}^2\,dv\right)^{\frac{1}{2}}
\lesssim \left(\int_{\mathbb{R}^3_v}w\norm{ \nabla^mf}_{L^2_x}^2\,dv\right)^{\frac{1-\theta}{2}}\left(\int_{\mathbb{R}^3_v}w\norm{ \nabla^\ell f}_{L^2_x}^2\,dv\right)^{\frac{\theta}{2}}
\end{equation}
where $0\le \theta\le 1$ and $\alpha$ satisfy
\begin{equation}
\frac{1}{p}-\frac{\alpha}{3}=\left(\frac{1}{2}-\frac{m}{3}\right)(1-\theta)+\left(\frac{1}{2}-\frac{\ell}{3}\right)\theta.
\end{equation}
\end{lemma}

\begin{proof}
For any function $f(x,v)$, by Lemma \ref{1interpolation}, we have
\begin{equation}\label{GN1}
\norm{\nabla^\alpha f}_{L^p_x}\lesssim \norm{  \nabla^mf}_{L^2_x}^{1-\theta}\norm{ \nabla^\ell f}_{L^2_x}^{\theta}.
\end{equation}
Taking the square of \eqref{GN1} and then multiplying by $w(v)$, integrating over $\mathbb{R}^3_v$, by H\"older's inequality, we obtain
\begin{equation}\label{GN2}
\begin{split}
\int_{\mathbb{R}^3_v}w\norm{\nabla^\alpha f}_{L^p_x}^2dv
&\lesssim \int_{\mathbb{R}^3_v}w\norm{  \nabla^mf}_{L^2_x}^{2(1-\theta)}\norm{ \nabla^\ell f}_{L^2_x}^{2\theta}\,dv
\\&=\int_{\mathbb{R}^3_v}\left(w^{\frac{1}{2} }\norm{\nabla^mf}_{L^2_x}\right)^{2(1-\theta)} \left(w^{\frac{1}{2} }\norm{\nabla^\ell f}_{L^2_x}\right)^{2\theta}\,dv
\\&\le\left(\int_{\mathbb{R}^3_v}\left(w^{\frac{1}{2} }\norm{\nabla^mf}_{L^2_x}\right)^{2} \,dv\right)^{1-\theta}\left(\int_{\mathbb{R}^3_v} \left(w^{\frac{1}{2} }\norm{ \nabla^\ell f}_{L^2_x}\right)^{2\theta}\,dv\right)^{\theta}.
\end{split}
\end{equation}
Taking the square root of \eqref{GN2}, we get \eqref{GN00}.
\end{proof}

\subsection{Negative Sobolev norms}

We define the operator $\Lambda^s, s\in \mathbb{R}$ by
\begin{equation}\label{1Lambdas}
\Lambda^s f(x)=\int_{\mathbb{R}^3}|\xi|^s\hat{f}(\xi)e^{2\pi ix\cdot\xi}\,d\xi,
\end{equation}
where $\hat{f}$ is the  Fourier transform of $f$. We define the homogeneous Sobolev space
$\dot{H}^s$ of all $f$ for which $\norm{f}_{\dot{H}^s}$  is finite, where
\begin{equation}\label{1snorm}
\norm{f}_{\dot{H}^s}:=\norm{\Lambda^s f}_{L^2}=\norm{|\xi|^s \hat{f}}_{L^2}.
\end{equation}
We will use the non-positive index $s$. For convenience, we will change the index to be ``$-s$" with $s\ge 0$. We will employ the following special Sobolev interpolation:
\begin{lemma}\label{1-sinte}
Let $s\ge 0$ and $\ell\ge 0$, then we have
\begin{equation}\label{1-sinterpolation}
\norm{\nabla^\ell f}_{L^2}\le \norm{\nabla^{\ell+1} f}_{L^2}^{1-\theta}\norm{\Lambda^{-s} f}_{L^2}^\theta, \hbox{ where }\theta=\frac{1}{\ell+1+s}.
\end{equation}
\end{lemma}
\begin{proof}
By the Parseval theorem, the definition of \eqref{1snorm} and H\"older's inequality, we have
\begin{equation}
\norm{\nabla^\ell f}_{L^2}
=\norm{|\xi|^\ell \hat{f}}_{L^2}\le  \norm{|\xi|^{\ell+1} \hat{f}}_{L^2}^{1-\theta}\norm{|\xi|^{-s} \hat{f}}_{L^2}^\theta
=\norm{\nabla^{\ell+1}f}_{L^2}^{1-\theta}\norm{\Lambda^{-s} f}_{L^2}^\theta.
\end{equation}
\end{proof}

We shall use the corresponding Sobolev interpolation for the functions on $\r3_x\times\r3_v$.
\begin{lemma}\label{-sinte}
Let $s\ge 0$ and $\ell\ge 0$, then we have
\begin{equation}
\norm{\nabla^\ell f}\lesssim \norm{\nabla^{\ell+1} f}^{1-\theta}\norm{\Lambda^{-s}f}^\theta, \hbox{ where }\theta=\frac{1}{\ell+1+s}.
\end{equation}
\end{lemma}

\begin{proof}
It follows by further taking the $L^2$ norm of \eqref{1-sinterpolation} over $\mathbb{R}_v^3$.
\end{proof}

If $s\in(0,3)$, $\Lambda^{-s}f$ defined by \eqref{1Lambdas} is the Riesz potential. The Hardy-Littlewood-Sobolev theorem implies the following $L^p$ inequality for the Riesz potential:
\begin{lemma}\label{1Riesz}
Let $0<s<3,\ 1<p<q<\infty,\ 1/q+s/3=1/p$, then
\begin{equation}\label{1Riesz es}
\norm{\Lambda^{-s}f}_{L^q}\lesssim\norm{ f}_{L^p}.
\end{equation}
\end{lemma}
\begin{proof}
See \cite[pp. 119, Theorem 1]{S}.
\end{proof}

\subsection{Minkowski's  inequality}

In estimating the nonlinear terms, it is important to use the Minkowski's integral inequality to interchange the orders of integration over $x$ and $v$.
\begin{lemma}\label{Minkowski}
Let $1\le p<\infty$. Let $f$ be a measurable function on $\mathbb{R}_y^3\times \mathbb{R}_z^3$, then we have
\begin{equation}\label{mink00}
\left(\int_{\mathbb{R}_z^3}\left(\int_{\mathbb{R}_y^3}|f(y,z)|\,dy\right)^p\,dz\right)^\frac{1}{p}
\le\int_{\mathbb{R}_y^3}\left(\int_{\mathbb{R}_z^3}|f(y,z)|^p\,dz\right)^\frac{1}{p}\,dy.
\end{equation}
In particular, for $1\le p\le q\le \infty$, we have
\begin{equation}\label{min es}
\norm{f}_{L^q_zL^p_y}\le \norm{f}_{L^p_yL^q_z}.
\end{equation}
\end{lemma}
\begin{proof}
The inequality \eqref{mink00} can be found in \cite[pp. 271, A.1]{S}, hence it remains to prove \eqref{min es}.
For $q=\infty$, we have
\begin{equation}
\norm{f}_{L^\infty_zL^p_y}=\sup_{z\in \r3}\left(\int_{\r3_y}|f|^p\,dy\right)^{1/p}\le \left(\int_{\r3_y}\left(\sup_{z\in
\r3}|f|\right)^p\,dv\right)^{1/p}=\norm{f}_{L^p_yL^\infty_z}.
\end{equation}
For $q<\infty$ and hence $1\le q/p<\infty$, then by \eqref{mink00}, we have
\begin{equation}
\norm{f}_{L^q_zL^p_y}= \left(\int_{\r3_z}\left(\int_{\r3_y}|f|^p\,dy\right)^{q/p}\,dz\right)^{1/q}\le
\left(\int_{\r3_y}\left(\int_{\r3_z}|f|^q\,dz\right)^{p/q}\,dv\right)^{1/p}=\norm{f}_{L^p_yL^q_z}.
\end{equation}
\end{proof}

\subsection{Boltzmann collision operators}

First, we collect some useful estimates of the linear collision operators.
\begin{lemma}\label{linear c} For $i=1,2$, we have
\begin{eqnarray}\label{1}&&\langle \mathcal{L}_i h_1,h_2\rangle=\langle
h_1,\mathcal{L}_ih_2\rangle,\ \langle \mathcal{L}_i h,h\rangle\ge 0,\\
\label{2}&&\mathcal{L}_i h=0\hbox{ if and only if }h={\bf P_i} h.
\end{eqnarray}
Moreover, there exist $\sigma_0>0$ and constants $C,C_{|\beta|}>0$
such that
\begin{eqnarray}\label{3}&&\langle\mathcal{L}_i
h,h\rangle\ge \sigma_0\left|\{{\bf I}-{\bf P_i}\}h\right|_\nu^2,\\
\label{4}&&\langle\nu\partial_\beta\mathcal{L}_i h,
\partial_\beta
 h\rangle\ge \frac{1}{2}\left|\nu\partial_\beta
 h\right|_2^2-C|h|_\nu^2,\\
\label{5}&&\langle\partial_\beta \mathcal{L}_i h,
\partial_\beta
 h\rangle\ge \frac{1}{2}\left|\partial_\beta
 h\right|_\nu^2-C_{|\beta|}\left|h\right|_\nu^2.\end{eqnarray}
\end{lemma}
\begin{proof}
We refer to \cite[Lemma 1]{G2003} or \cite[Lemma 3.1]{W} for \eqref{1}--\eqref{2}. While for the proof of \eqref{3}--\eqref{5}, we refer to \cite[Lemma
3.2--3.3]{G2006}.
\end{proof}

Now, we collect some useful estimates of the nonlinear collision operator.
\begin{lemma}\label{nonlinearcol1}
There exists $C>0$ such that
\begin{equation}\label{nonlineares2}
|\langle\Gamma(h_1,h_2),h_3\rangle|+|\langle\Gamma(h_2,h_1), h_3\rangle|\le  C\sup_{v}\{\nu^3 h_3\}|h_1|_2|h_2|_2.
\end{equation}
Moreover, for any $0\le \eta \le 1$, we have
\begin{equation}\label{nonlineares3}
|\nu^{-\eta}\Gamma(h_1,h_2)|_2\le C\left\{|\nu^{1-\eta} h_1|_2|h_2|_2+|\nu^{1-\eta}h_2|_2|h_1|_2\right\}.
\end{equation}
\end{lemma}

\begin{proof}
We refer to \cite[Lemma 2.3]{G2002} for \eqref{nonlineares2}, and \cite[Lemma 2.7]{UY2006} for \eqref{nonlineares3}.
\end{proof}

\end{document}